\documentclass[11pt,a4paper]{article}
\usepackage[margin=0.75in]{geometry}
\usepackage{authblk}
\usepackage{parskip}
\usepackage{amsmath, amssymb, amsfonts, caption, enumitem, graphicx, mathtools, nicefrac, xcolor, comment, subcaption, dsfont} 
\usepackage[square,sort,numbers]{natbib}
\usepackage[title]{appendix}
\usepackage[utf8]{inputenc}
\usepackage[T1]{fontenc}
\usepackage{hyperref}
\usepackage[dvipsnames]{xcolor}
\hypersetup{colorlinks, linkcolor={red!50!black}, citecolor={blue!50!black}, urlcolor={ForestGreen}}

\usepackage{amsthm}
\newtheorem{theorem}{Theorem}

\newtheorem{lemma}{Lemma}
\newtheorem{corollary}{Corollary}

\theoremstyle{definition}

\newtheorem{definition}{Definition}
\newtheorem{assumption}{Assumption}
\newtheorem{condition}{Condition}

\usepackage[capitalize]{cleveref}
\crefname{assumption}{Assumption}{Assumptions}
\Crefname{assumption}{Assumption}{Assumptions}
\crefformat{equation}{(#2#1#3)}
\Crefname{condition}{Condition}{Conditions}
\crefformat{condition}{Condition #2#1#3}
\crefformat{assumption}{Assumption #2#1#3}
\crefrangeformat{equation}{(#3#1#4)--(#5#2#6)}
\crefrangeformat{condition}{Conditions #3#1#4--#5#2#6}
\crefrangeformat{lemma}{Lemmas #3#1#4--#5#2#6}
\crefrangeformat{algorithm}{Algorithms #3#1#4--#5#2#6}
\crefrangeformat{section}{Sections #3#1#4--#5#2#6}

\usepackage{anyfontsize}
\usepackage{t1enc}
\newcommand{\midsmall}{\fontsize{8.75pt}{10pt}\selectfont}
\usepackage{algorithmic}
\newcounter{algorithm}
\renewcommand{\thealgorithm}{\arabic{algorithm}}
\newlength{\algcolht}
\newenvironment{algobox}[2]{%
  \refstepcounter{algorithm}%
  \par\noindent\begin{minipage}[t]{\linewidth}\midsmall
  \hrule height 0.5pt\relax
  \vspace{2pt}
  \textbf{Algorithm \thealgorithm}~#1\label{#2}\par
  \vspace{1.5pt}
  \hrule height 0.5pt\relax
  \vspace{3pt}
  \begin{algorithmic}
}{%
  \end{algorithmic}
  \vspace{2pt}
  \hrule height 0.5pt\relax
  \end{minipage}\par
}

\newcommand{\alginit}[1]{\STATE \textbf{Initialize:} #1}
\newcommand{\alginput}[1]{\STATE \textbf{Input:} #1}

\definecolor{gold}{rgb}{0.85,0.65,0}
\colorlet{dgreen}{green!60!black}

\def\beq{\begin{equation}}
\def\eeq{\end{equation}}

\def\fnote#1{\footnote}

\newcommand{\grad}{\ensuremath{\nabla}}

\def\R{{\mathbb{R}}}

\def\cA{{\cal A}}

\def\cD{{\cal D}}

\def\cG{{\cal G}}
\def\cH{{\cal H}}
\def\cI{{\cal I}}

\def\cN{{\cal N}}
\def\cO{{\cal O}}

\def\cS{{\cal S}}

\def\cX{{\cal X}}

\newcommand{\bbP}{\mathbb{P}}

\newcommand{\bbS}{\mathbb{S}}

\newcommand{\bbZ}{\mathbb{Z}}

\DeclareMathOperator*{\argmin}{arg\,min}
\DeclareMathOperator*{\argmax}{arg\,max}

\DeclareMathOperator{\Tr}{Tr}

\DeclareMathOperator{\lin}{lin}

\DeclareMathOperator{\conv}{conv}

\DeclarePairedDelimiter\abs{\lvert}{\rvert}%

\makeatletter
\let\oldabs\abs
\def\abs{\@ifstar{\oldabs}{\oldabs*}}
\makeatother

\title{Auto-Conditioned Frank-Wolfe Algorithms}

\author[1]{Khanh-Hung Giang-Tran}
\author[1]{Soroosh Shafiee}
\author[2]{Nam Ho-Nguyen}

\affil[1]{Operations Research and Information Engineering, Cornell University, USA \protect\\ \texttt{\{tg452,shafiee\}@cornell.edu}}
\affil[2]{Discipline of Business Analytics, The University of Sydney, Sydney, Australia \protect \\ \texttt{nam.ho-nguyen@sydney.edu.au}}

\date{}
\begin{document}

\maketitle

\begin{abstract}
Frank-Wolfe methods are projection-free algorithms for constrained optimization whose practical performance often depends critically on the choice of step size. Classical closed-loop step-size rules typically require prior knowledge of a global smoothness constant, while line-search variants avoid this requirement at the cost of additional function evaluations and implementation overhead. In this paper, we develop a  fully auto-conditioned framework for Frank-Wolfe-type methods. The framework replaces the global Lipschitz constant in closed-loop step sizes with a local Lipschitz estimator computed from first-order information along the iterates. We show that this abstraction captures several important projection-free subroutines, including standard Frank-Wolfe, Matching Pursuit, pairwise Frank-Wolfe, and away-step Frank-Wolfe. For the resulting general class of methods, we establish convergence to stationary points in the nonconvex setting and recover the standard sublinear convergence guarantees in the convex setting, without requiring prior knowledge of a global smoothness constant. We further show that, when specialized to particular Frank-Wolfe variants and combined with additional structural assumptions, the same auto-conditioned framework yields accelerated convergence rates. Numerical experiments demonstrate that the proposed methods provide substantial practical improvements over line-search-based alternatives, highlighting the benefits of adapting to local curvature while retaining the simplicity of closed-loop step-size rules.
\end{abstract}

\section{Introduction}
Let $\cX \subset \R^n$ be a given feasible set, and consider the optimization problem
\begin{equation} \label{eq:opt}
    \min_{x \in \cX} ~ f(x).
\end{equation}
We denote $x^\star$ to be an optimal solution to~\eqref{eq:opt}.
In this paper we are interested in settings where the feasible set $\cX$ is constructed via a dictionary $\cA$, and the objective function satisfies standard smoothness assumptions.
\begin{assumption}[Dictionary representation]\label{ass:X-compact}
    There exists a dictionary $\cA \subset \R^n$ such that $\cX = \conv(\cA)$ or $\cX = \lin(\cA)$ where $\conv(\cdot)$ denotes the convex hull of a set, and $\lin(\cdot)$ denotes its linear span.
    Moreover, the dictionary $\cA$ is compact set with diameter $D_{\cA} < \infty$, that is, 
    $$\| x - y \|_2 \leq D_{\cA}, \qquad \forall x,y \in \cA.$$
\end{assumption}
\begin{assumption}[Lipschitz gradient]\label{ass:smooth-functions}
    $\grad f: \R^n \to \R^n$ is $L$-Lipschitz continuous, that is,
    $$\|\grad f(x) - \grad f(y)\|_2 \leq L \|x-y\|_2, \quad \forall x,y \in \R^n.$$    
\end{assumption}
In settings where \cref{ass:X-compact} holds, a common consideration is to look for approximate solutions to \cref{eq:opt} that are \emph{sparse} in the dictionary $\cA$. For $\conv(\cA)$ and $\lin(\cA)$, under \cref{ass:smooth-functions} it is well-known that Frank-Wolfe and Matching Pursuit-type methods can find sparse solutions iteratively.

A key part of such methods is the selection of the stepsize, i.e., a parameter that governs how far we move away from the current iterate. A common choice is the classic closed-loop stepsize rule (also referred to as the short-step rule) based on the smoothness constant $L$. Namely, at iteration $t$, given $x_t \in \cX$, one sets the stepsize to be 
\begin{align}
    \label{eq:closed-loop}
    \gamma_t \gets \argmin_{\gamma \in [0,\gamma_t^{\max}]} \left\{ f(x_t) - \gamma \grad f(x_t)^\top d_t + \frac{\gamma^2 L \|d_t\|_2^2}{2} \right\} = \min\!\left\{\dfrac{\grad f(x_t)^\top d_t}{L \lVert d_t \rVert_2^2},\, \gamma_t^{\max} \right\},
\end{align}
where $\gamma_t^{\max}$ the maximal admissible stepsize and $d_t$ is the search direction; we will discuss how exactly these are set later. 
For example, under the closed-loop Frank-Wolfe subroutine, one has $\gamma_t^{\max}=1$.
Note that the equality in \cref{eq:closed-loop} follows from the fact that the search direction $d_t$ will be chosen to ensure $\grad f(x_t)^\top d_t \geq 0$ for all $t$.
Under appropriate conditions, a closed-loop stepsize rule can yield accelerated convergence guarantees compared to more conservative open-loop (i.e., fixed schedule) stepsizes. However, implementing it requires knowledge of $L$, and even when an estimate is available, its practical performance depends strongly on the quality of the estimate. That is, an overly large estimate typically leads to excessively small steps and therefore slow optimization progress.

An effective strategy to encourage less conservative closed-loop stepsizes is to use local smoothness parameters instead of the global constant $L$. While these are rarely available a priori, line search allows us to test different stepsizes to look for a sufficient decrease condition that would replicate the knowledge of such parameters. The disadvantage of this, however, is that each iteration may require multiple objective evaluations before an acceptable stepsize is found, which can substantially increase the per-iteration cost. In settings where these evaluations are expensive, the practical gains from taking larger steps may therefore be offset by the additional computational overhead of the line search procedure.

The main goal of this paper is to develop an auto-conditioned (line search-free) mechanism that improves this closed-loop rule by replacing the global parameter $L$ with a local Lipschitz estimator computed from the current trial point, and provide convergence analyses for both Frank-Wolfe- and Matching Pursuit-type methods in a unified framework. To this end, for any $x,y\in\R^n$, we define
\begin{equation}\label{eq:ell-def}
    \ell(x,y) :=
    \begin{cases}
        \dfrac{2\left|f(y)-f(x)-\grad f(x)^\top (y-x)\right|}{\|y-x\|_2^2}, & x \neq y,\\[1.25ex]
        0, & x=y.
    \end{cases}
\end{equation}
This quantity serves as a local estimate of the Lipschitz of $f$ along the segment joining $x$ and $y$, and it will be the basis for our auto-conditioned stepsize strategy.

The key contributions of this paper are summarized below.
\begin{enumerate}[label=$\diamond$,leftmargin=*]
    \item We propose a general abstract framework for auto-conditioned Frank-Wolfe methods based on a closed-loop step-size rule, where it uses the local Lipschitz estimator introduced in~\eqref{eq:ell-def}. We further show that several Frank-Wolfe variants, including the closed-loop Frank-Wolfe, standard Matching Pursuit, pairwise Frank-Wolfe, and away-step Frank-Wolfe subroutines, fit this abstraction.

    \item For this general framework, we establish convergence to stationary points in the nonconvex setting. In the convex setting, we recover the known sublinear convergence guarantees usually derived under prior knowledge of a global Lipschitz constant with closed-loop stepsizes, but now for a fully auto-conditioned method that uses only local Lipschitz estimates. We then show that acceleration is possible when the framework is specialized to particular subroutines. Namely, we establish accelerated convergence rates for the corresponding Frank-Wolfe variants under additional structural assumptions.

    \item Finally, our numerical results show that the proposed auto-conditioned approach is competitive with line-search alternatives and offers improvements in regimes where function evaluations are expensive.
\end{enumerate}

\subsection{Related Works}

\paragraph{Auto-Conditioned (Line Search-Free) Methods.} While adaptive methods also exist for nonsmooth objectives, we focus here on smooth minimization and closely related problems, since these are the most relevant comparators under our assumptions. 
AdaGrad~\citep{duchi2011adaptive} pioneered this line of work and has since inspired a rich family of adaptive methods~\citep{kingma2015adam,reddi2018convergence,tieleman2012lecture,zeiler2012adadelta,wang2026universal}. A central theme is transferring adaptive online learning techniques to offline settings, extending projected gradient-type methods to offline convex optimization~\citep{levy2017online,levy2018online,cutkosky2019anytime,kavis2019unixgrad}, monotone variational inequalities~\citep{bach2019universal}, and nonconvex optimization~\citep{attia2023sgd,alacaoglu2021convergence,kavis2022high,wang2023convergence,ward2020adagrad}. In each case, adaptivity is driven by cumulative gradient- or discrepancy-based statistics accumulated over the run, yielding a globally safe learning-rate schedule that matches the guarantees of parameter-tuned methods without prior knowledge of smoothness or noise levels. A limitation of cumulative-based methods, however, is that their stepsizes decrease monotonically, preventing the algorithm from exploiting favorable local curvature to accelerate progress at particular iterations.

A second line of work moves away from online-to-offline conversion, and instead forms an estimate of the local smoothness constant at each iteration by using two past iterates, e.g., via \cref{eq:ell-def} or a variant of it. These estimates allow the stepsizes to increase, though a clipping procedure is typically applied to prevent wild fluctuations and to facilitate meaningful analysis. The earliest results are due to \citet{malitsky2020adaptive,malitsky2024adaptive}, who treat the convex (proximal) gradient setting without acceleration; the analysis has since been extended to local Lipschitz continuity of the gradient \citep{latafat2025adaptive} and to the Bregman proximal setting without relative smoothness \citep{ou2025linesearch}. Building on this template, accelerated convergence rates have been established \citep{li2025simple,suh2025adaptive,wang2025adaptive}, along with guarantees for nonconvex problems \citep{lan2024projected,ye2025simple}. Closely related ideas have proven useful in a broad range of structured settings, including primal-dual methods and ADMM-type algorithms \citep{vladarean2021first,lan2024auto,jang2026alia}, variational inequalities \citep{malitsky2020golden,alacaoglu2023beyond,shen2026parameter} and Riemannian optimization~\citep{ansari2025adaptive,park2026adaptive}. This second line of work is the inspiration for our paper, where we aim to bring these ideas to Frank-Wolfe-type methods and obtain convex, accelerated, and nonconvex convergence guarantees.

\paragraph{Closed-Loop Frank-Wolfe.} 
The Frank-Wolfe algorithm was originally introduced in \citep{frank1956algorithm} for quadratic programming, and was later generalized far beyond this initial setting to constrained convex problems for which the feasible region is accessible through an efficient linear minimization oracle \citep{jaggi2013revisiting,lacoste2013block,beck2015cyclic, braun2022conditional}. Early convergence analyses of the classical closed-loop step-size rule in \eqref{eq:closed-loop} can be traced back to the work of \citet{dunn1980convergence}. Beyond the basic sublinear rate of $\cO(1/T)$ after $T$ iterations, accelerated rates for Frank-Wolfe methods with closed-loop stepsizes have been established under additional geometric or structural assumptions, going back to \citet{polyak1966constrained} and more recently refined in works such as \citet{garber2015faster}. More recently, \citet{pedregosa2020linearly} proposed line-search and backtracking variants that adapt to unknown smoothness constants while maintaining these improved convergence rates. 
In this paper, we revisit variants of the Frank-Wolfe algorithms through an auto-conditioned scheme based on the local Lipschitz estimator in~\eqref{eq:ell-def}. One closely-related prior work is \citep{yuan2026adaptive}, which develops an adaptive Frank-Wolfe method, with a focus on the stochastic nonconvex setting. Their estimator, however, is AdaGrad-style \citep{levy2017online,levy2018online} and it accumulates gradient information across iterations and is therefore monotonically non-decreasing, leading to conservative stepsizes. 
Another closely related work is \citep{yagishita2025simple}, which relies on the local estimator~\eqref{eq:ell-def}. Nevertheless, \citep[Algorithm~3]{yagishita2025simple} does not include the damping factor $r_t$, making the estimator non-decreasing. Moreover, \citep{yagishita2025simple} analyzes only the nonconvex setting, and only for the closed-loop Frank-Wolfe algorithm. Our algorithm, by contrast, allows non-monotonic stepsizes, and we establish convergence guarantees for a range of Frank-Wolfe variants in both the convex and nonconvex settings, including accelerated rates under additional assumptions.

\paragraph{Standard Matching Pursuit.}
Matching Pursuit, introduced by \citet{mallat1993matching}, can be viewed as the counterpart of Frank-Wolfe for optimization over $\lin(\cA)$. 
This connection was made explicit by \citet{locatello2017unified}, who developed a unified optimization view of generalized Matching Pursuit and Frank-Wolfe and obtained affine-invariant sublinear and linear convergence guarantees. 
More recently, \citet{pedregosa2020linearly} proposed line-search and backtracking variants that preserve these improved rates in the adaptive setting. In this paper, we revisit the standard Matching Pursuit method through a auto-conditioned scheme and show that it recovers the standard and improved guarantees without requiring prior knowledge of a global Lipschitz constant. While our framework could potentially be extended further, we do not consider here other variants such as the conic extension~\citep{locatello2018greedy}, orthogonal Matching Pursuit~\citep{pati1993orthogonal}, or blended variants~\citep{combettes2019blended}.

\paragraph{Pairwise Frank-Wolfe.}
The pairwise Frank-Wolfe method updates the iterate by transferring mass from an active atom to a newly selected Frank-Wolfe atom, thereby reducing the zig-zagging behavior that often slows down the classical Frank-Wolfe method near the boundary of the feasible set \citep{canon1968tight}. 
Under strong convexity and suitable geometric assumptions on the constraint set, pairwise Frank-Wolfe methods enjoy global linear convergence guarantees \citep{lacoste2015global}. \citet{pedregosa2020linearly} proposed line-search and backtracking variants that preserve these improved rates. 
In this paper, we revisit the pairwise Frank-Wolfe method through an auto-conditioned scheme and show that it recovers the improved guarantees.

\paragraph{Away-Step Frank-Wolfe.}
Away-step Frank-Wolfe, formalized by \citet{guelat1986some}, is a classical variant of Frank-Wolfe that allows the algorithm to move away from atoms in the current active-set representation. 
Over polytope constraint set, linear convergence of the away-step variant was established by \citet{guelat1986some} under a strict complementarity-type assumption, and later sharpened and extended in affine-invariant and global forms by \citet{lacoste2015global}.
Related analyses have extended this picture in several directions, including non-strongly-convex composite objectives in the work of \citet{beck2017linearly} and geometry- and conditioning-based interpretations in the work of \citet{pena2016neumann, pena2019polytope}. In this paper, we revisit the away-step Frank-Wolfe method through an auto-conditioned scheme and show that it recovers the standard and improved guarantees without requiring prior knowledge of a global Lipschitz constant.

\subsection{Notation and Outline}
For a set $\cA$, we denote by $\lin(\cA)$ its linear span and by $\conv(\cA)$ its convex hull. We denote by $\bbZ_+$ the set of nonnegative integers. For any $T \in \bbZ_+$, we write $[T] := \{0, 1, \dots, T \}$.
For an index set $\cI \subseteq \bbZ_+$, we define $\cI^c := \bbZ_+ \setminus \cI$. We denote by $\{x_{k_t}\}_{t \geq 0}$ the subsequence corresponding to an increasing sequence $\{k_t\}_{t \geq 0}$. We write $\mathds{1}_{\{\mathrm P\}}$ for the indicator of a statement $\mathrm P$, namely, $\mathds{1}_{\{\mathrm P\}} = 1$ if $\mathrm P$ holds; $= 0$ otherwise.

The remainder of the paper is organized as follows. Section~\ref{sec:framework} introduces the proposed auto-conditioned framework, presents the main outer-loop algorithm, and shows that several standard direction-finding rules fit into this framework. Section~\ref{sec:convergence} establishes the general convergence guarantees, covering both nonconvex and convex setting. Section~\ref{sec:acceleration} then develops accelerated rates for particular subroutines under stronger geometric and convexity assumptions.
Section~\ref{sec:numerical} concludes the paper with numerical experiments, showcasing the strength of the proposed framework. 

\section{Proposed Framework}\label{sec:framework}

We now introduce an adaptive variant of the classical closed-loop rule. 
The key difference is that the global Lipschitz constant $L$ in~\eqref{eq:closed-loop} is replaced by a local estimate $L_t$, updated recursively along the iterates. 
This yields an auto-conditioned framework that preserves the basic structure of Frank-Wolfe and Matching Pursuit methods while making the stepsize less sensitive to global smoothness estimates.

\cref{alg:ac-fw} provides the common outer-loop framework used throughout this paper. At iteration $t$, $\cS_t \subset \cA$ denotes the current active set of dictionary vectors, and $\alpha_{\cdot,t}$ denotes the vector of corresponding weights that comprise $x_t$, so that we maintain the active set representation $x_t = \sum_{s \in \cA} \alpha_{s,t} s$, and $\alpha_{s,t} = 0$ iff $s \not\in \cS_t$. The algorithm then computes $v_t$ to be the minimizer of $\grad f(x_t)^\top v$ over $v \in \cA$, and then calls a direction-finding subroutine, denoted by $\texttt{subrout}(v_t,x_t,\cS_t,\alpha_{\cdot,t})$, that returns a search direction~$d_t$ together with a maximal admissible stepsize $\gamma_t^{\max}$. The method then forms a trial point $\bar x_{t+1}$ using the current estimate $L_t$, updates this estimate through the local Lipschitz estimator \eqref{eq:ell-def}, and accepts the trial point only if the new estimate remains sufficiently controlled relative to $L_t$. While this acceptance step is reminiscent of a stepsize acceptance rule used in line  search methods, our framework is distinct from line search. The reason is that even for iterations where the new trial point $\bar{x}_{t+1}$ is \emph{not} accepted, we still end the iteration. This limits the number of function evaluations each iteration to $1$. In contrast, a line search rule may involve multiple function evaluations each iteration.

This abstraction allows \cref{alg:ac-fw} to cover several classical Frank-Wolfe-type variants through different choices of subroutine. 
In particular, \cref{alg:fw-sub} corresponds to the closed-loop Frank-Wolfe direction, and \cref{alg:mp-sub} corresponds to the standard Matching Pursuit direction. 
For these two variants, no additional state needs to be maintained, so $\cS_t,\alpha_{\cdot,t}$ may be ignored. 
In contrast, in the pairwise and away-step variants, given respectively in \cref{alg:pairwise-sub,alg:away-sub}, the subroutine must keep track of an active-set representation of the current iterate $(\cS_t,\alpha_{\cdot,t})$.

\begin{figure*}[!bt]
	\centering
	
	\begin{minipage}[t]{0.325\textwidth}
		\raggedright
		\begin{algobox}{\texttt{AC-FW} Algorithm}{alg:ac-fw}
			\alginput{damping sequence $\{r_t\}_{t\ge 0}$\!\!\!} \vspace{0.5ex}
			\alginit{$x_{-1}\in\cA$}
			\STATE $x_0 \in \argmin_{v\in\cA} \grad f(x_{-1})^\top v$
			\STATE $L_0 \gets \ell(x_{-1},x_0)$ \vspace{0.5ex}
			\FOR{$t=0,1,2,\ldots$} \vspace{0.5ex}
			\STATE $v_t \in \argmin_{v\in\cA} \grad f(x_t)^\top v$ \vspace{0.75ex}
			\STATE $(d_t,\gamma_t^{\max}) \!\gets\! \texttt{subrout}(v_t, x_t, \cS_t,\alpha_{\cdot,t})$\!\! \vspace{-1.4ex}
			\STATE $\gamma_t \gets \min\!\left\{\dfrac{\grad f(x_t)^\top d_t}{L_t\lVert d_t\rVert_2^2},\,\gamma_t^{\max}\right\}$ \vspace{0.65ex}
			\STATE $\bar x_{t+1} \gets x_t-\gamma_t d_t$ \vspace{0.5ex}
			\STATE $L_{t+1} \gets \max\!\{\ell(x_t,\bar x_{t+1}),\, r_tL_t\}$ \vspace{0.5ex}
			\IF{$f(\bar x_{t+1})<f(x_t)$}
			\STATE $x_{t+1} \gets \bar x_{t+1}$
			\ELSE
			\STATE $x_{t+1} \gets x_t$
			\ENDIF
			\ENDFOR
		\end{algobox}
	\end{minipage}
	\hfill
	\begin{minipage}[t]{0.285\textwidth}
		\raggedright
		
		\begin{algobox}{\texttt{CFW-subroutine}}{alg:fw-sub}
			\STATE $d_t \gets x_t-v_t, \ \gamma_t^{\max} \gets 1$
		\end{algobox}
		
		\vspace{1ex}
		
		\begin{algobox}{\texttt{MP-subroutine}}{alg:mp-sub}
			\STATE $d_t \gets -v_t, \ \gamma_t^{\max} \gets \infty$
		\end{algobox}
		
		\vspace{1ex}
		
		\begin{algobox}{\texttt{PFW-subroutine}}{alg:pairwise-sub}
			\alginit{$(\cS_0,\alpha_{\cdot,0}) \gets (\{x_0\}, \! 1)$ \!\!\!\!\!} \vspace{0.5ex}
			\STATE $s_t \in \argmax_{s\in\cS_t} \grad f(x_t)^\top s$ \vspace{0.5ex}
			\STATE $d_t \gets s_t-v_t, \ \gamma_t^{\max} \gets \alpha_{s_t,t}$ \vspace{1.1ex}
			\IF{$x_{t+1}=\bar x_{t+1}$}
			\STATE $\alpha_{s_t,t+1} \gets \alpha_{s_t,t}-\gamma_t$
			\STATE $\alpha_{v_t,t+1} \gets \alpha_{v_t,t}+\gamma_t$
			\STATE $\alpha_{s,t+1} \gets \alpha_{s,t}$ for $s\notin\{s_t,v_t\}$
			\STATE $\cS_{t+1} \gets \{s\in\cA:\alpha_{s,t+1}>0\}$
			\ELSE
			\STATE $(\cS_{t+1},\alpha_{\cdot,t+1})\gets(\cS_t,\alpha_{\cdot,t})$
			\ENDIF
		\end{algobox}
		
	\end{minipage}
	\begin{minipage}[t]{0.375\textwidth}
		\raggedright
		\begin{algobox}{\texttt{AFW-subroutine}}{alg:away-sub}
			\alginit{$(\cS_0,\alpha_{\cdot,0})\gets(\{x_0\},1)$} \vspace{0.5ex}
			\STATE $s_t \in \argmax_{s\in\cS_t} \grad f(x_t)^\top s$ \vspace{0.5ex}
			\IF{$\grad f(x_t)^\top(x_t-v_t) \!\ge\! \grad f(x_t)^\top(s_t-x_t)$}
			\STATE $d_t\gets x_t-v_t$, \ $\gamma_t^{\max}\gets 1$, \ $\texttt{step} \gets \texttt{FW}$
			\ELSE
			\STATE $d_t \gets s_t-x_t$,  $\gamma_t^{\max} \gets  \tfrac{\alpha_{s_t,t}}{1 - \alpha_{s_t,t}}$,  $\texttt{step} \gets  \texttt{AW}$ 
			\ENDIF
			\IF{$x_{t+1}=x_t$}
			\STATE $(\cS_{t+1},\alpha_{\cdot,t+1})\gets(\cS_t,\alpha_{\cdot,t})$
			\ELSIF{$\texttt{step}=\texttt{FW}$}
			\STATE $\alpha_{s,t+1}\gets (1-\gamma_t)\alpha_{s,t}$ for $s\neq v_t$
			\STATE $\alpha_{v_t,t+1}\gets (1-\gamma_t)\alpha_{v_t,t}+\gamma_t$
			\STATE $\cS_{t+1}\gets \{s\in\cA:\alpha_{s,t+1}>0\}$
			\ELSE
			\STATE $\alpha_{s,t+1}\gets (1+\gamma_t)\alpha_{s,t}$ for $s\neq s_t$
			\STATE $\alpha_{s_t,t+1}\gets (1+\gamma_t)\alpha_{s_t,t}-\gamma_t$
			\STATE $\cS_{t+1}\gets \{s\in\cA:\alpha_{s,t+1}>0\}$
			\ENDIF
		\end{algobox}
	\end{minipage}
\end{figure*}

We define two index sets associated with \cref{alg:ac-fw}.
\begin{definition}\label{def:index-sets}
For a fixed constant $\eta>1$, the index set of \emph{significant-descent iterations} is defined by
\begin{align}
    \label{eq:cI}
    \cI_\eta := \{t\geq 0 : L_{t+1} \leq \eta L_t\}.
\end{align}
The index set of \emph{good iterations} is defined by
\begin{align}\label{eq:cG}
\cG:=\{t \geq 0: \gamma_t^{\max} \geq 1 \ \text{or} \ \gamma_t < \gamma_t^{\max}\}.
\end{align}
\end{definition}

The analysis below relies on two abstract requirements. 
The first concerns the damping
sequence $\{r_t\}_{t \geq 0}$ used in updating the estimates $L_t$, while the second concerns the chosen direction-finding subroutine.

\begin{condition}[Damping schedule]
\label{cond:r_t} 
    The damping factor satisfies $r_t \in (0,1]$, and the infinite product satisfies $r:=\prod_{t\geq 0} r_t \in (0, 1]. $
\end{condition}

\cref{cond:r_t} is mild as it essentially requires the damping factors to approach $1$
sufficiently fast so that the cumulative shrinkage remains finite.
An admissible choice for the damping factor, which we employ throughout our numerical experiments, is given by
\begin{align}
    \label{eq:r_t}
    r_t = 1-\frac{1}{(t+1)\log^{1+\delta}(t+3)},
    \qquad \forall t\geq 0,
\end{align}
for some $\delta>0$.
\begin{lemma}\label{lemma:damping}
    Let $\{r_t\}_{t \geq 0}$ be given by \eqref{eq:r_t}. Then, \cref{cond:r_t} holds.
\end{lemma}

\begin{proof}
    For every $t\geq 0$, we clearly have $r_t \in (0,1)$.
    Moreover,
    \begin{align*}
        \sum_{t\geq 0}(1-r_t)
        = \sum_{t\geq 0}\frac{1}{(t+1)\log^{1+\delta}(t+3)}
        < \infty,
    \end{align*}
    since this series is comparable to the integral $\int_3^\infty \frac{dx}{x\log^{1+\delta}x}$, which is finite for $\delta > 0$.
    Since $r_t\in(0,1)$ and $r_t \to 1$, this implies $\sum_{t\geq 0}-\log(r_t)<\infty$, and therefore the infinite product $\prod_{t \geq 0} r_t$ converges to a strictly positive limit. 
    Thus, $r:=\prod_{t\geq 0} r_t \in (0,1]$, so \cref{cond:r_t} holds.
\end{proof}
We next show that the adaptive Lipschitz estimates produced by \cref{alg:ac-fw} remain uniformly bounded by the global smoothness constant.

\begin{lemma}\label{lemma:Lipschitz}
    Suppose $\{x_t\}_{t\geq 0}$ is the sequence generated by \cref{alg:ac-fw}. If \cref{ass:smooth-functions} holds and $r_t \leq 1$ for all $t$, then $L_t \leq L$ for any $t \geq 0$.
\end{lemma}

\begin{proof}[Proof of \cref{lemma:Lipschitz}]
    We prove by induction on $t$.
    By definition, $L_0 = \ell(x_{-1}, x_0).$
    Since $f$ has $L$-Lipschitz gradient thanks to \cref{ass:smooth-functions}, the standard descent inequality gives
    \begin{align*}
        \left| f(y) - f(x) - \grad f(x)^\top (y-x) \right|
        \leq \frac{L}{2} \| y - x \|_2^2,
        \qquad \forall x, y \in \R^n.
    \end{align*}
    Applying this with $(x,y)=(x_{-1},x_0)$ yields $L_0\leq L$.
    
    Now assume that $L_t\leq L$ for some $t\geq 0$. Applying the same inequality with $(x,y)=(x_t,\bar x_{t+1})$, we obtain
    \[
        \ell(x_t,\bar x_{t+1})
        =
        \frac{2\big|f(\bar x_{t+1})-f(x_t)-\grad f(x_t)^\top(\bar x_{t+1}-x_t)\big|}
        {\|\bar x_{t+1}-x_t\|_2^2}
        \leq L.
    \]
    Moreover, by \ref{cond:r_t}, we have $r_tL_t\leq L_t\leq L$.
    Therefore, $L_{t+1} = \max\{\ell(x_t,\bar x_{t+1}),\,r_tL_t\} \leq L$.
    This closes the induction and proves the claim.
\end{proof}

The next lemma establishes that significant-descent iterations occur infinitely often for \cref{alg:ac-fw} and bounds the gap between consecutive indices in $\cI_\eta$.

\begin{lemma} \label{lemma:index}
    Suppose $\{x_t\}_{t \geq 0}$ is the sequence generated by \cref{alg:ac-fw}, let $\eta>1$, and assume
    $L_0>0$. If \cref{ass:smooth-functions} and \cref{cond:r_t} hold, then non-significant-descent iterations only occur finitely many times, i.e., $|\cI_\eta^{\mathrm{c}}| < \infty$ where the index set $\cI_\eta$, is defined in~\eqref{eq:cI}. (Consequently, $|\cI_\eta| = \infty$.)
    Moreover, denoting the elements of $\cI_\eta$ increasing order by $\{k_t\}_{t \geq 0}$, it holds that for every $t\geq 0$,
    \begin{equation}\label{eq:index}
        t \leq k_t \leq t + \left\lfloor \log_\eta\!\left(\frac{L}{rL_0}\right) \right\rfloor.
    \end{equation}
\end{lemma}

\begin{proof}[Proof of \cref{lemma:index}]
    Since $L_{t+1} = \max\{\ell(x_t, \bar x_{t+1}), r_t L_t\}$, we have ${L_{t+1}}/{L_t} \geq r_t$ for all $t \geq 0$.
    Moreover, if $t \notin \cI_\eta$, then by definition $L_{t+1}> \eta L_t$, and therefore ${L_{t+1}}/{L_t} > \eta$. 
    Hence, for any $T \geq 0$ and letting $[T] := \{0, \dots, T \}$, we have
    \begin{align*}
        \frac{L_{T+1}}{L_0} = \prod_{t = 0}^{T} \frac{L_{t+1}}{L_t} 
        &= \left(\prod_{t \in [T] \cap \cI_\eta} \frac{L_{t+1}}{L_t}\right) \left(\prod_{t \in [T] \cap \cI_\eta^c} \frac{L_{t+1}}{L_t}\right)\\
        &\geq \left(\prod_{t \in [T] \cap \cI_\eta} r_t \right) \left(\prod_{t \in [T] \cap \cI_\eta^c} \eta \right) \\
        &\geq \left(\prod_{t \in [T]} r_t\right) \eta^{\,|[T] \cap \cI_\eta^c|}\\
        &\geq r \, \eta^{\,|[T] \cap \cI_\eta^c|}.
    \end{align*}
    By \cref{lemma:Lipschitz}, we have $L_{T+1} \leq L$ for all $T \geq 0$.
    Thus, $r \, \eta^{\,|[T] \cap \cI_\eta^c|} \leq {L_{T+1}}/{L_0} \leq {L}/{L_0}.$
    Therefore,
    \[
        |[T]\setminus \cI_\eta|
        =
        |[T]\cap \cI_\eta^c|
        \leq
        \log_\eta\!\left(\frac{L}{rL_0}\right),
        \qquad \forall T \geq 0.
    \]
    This implies that $\cI_\eta^c$ is finite, and hence $\cI_\eta$ is infinite, and the increasing sequence $\{k_t\}_{t\geq 0}$ is well defined.

    Now fix $t\geq 0$. Since $k_t$ is the $t$-th element of an increasing sequence of nonnegative integers, we clearly have $k_t \geq t$. On the other hand, we have $\left| \cI_\eta \cap [k_t] \right| = t+1$.
    Therefore,
    \begin{align*}
        k_t+1
        = \left| \cI_\eta\cap[k_t] \right| + \left| \cI_\eta^c \cap [k_t] \right|
        = (t+1) + \left| \cI_\eta^c \cap [k_t] \right| 
        \leq t+1+\left\lfloor \log_\eta\!\left(\frac{L}{rL_0}\right)\right\rfloor.
    \end{align*}
    Thus, the claim in \eqref{eq:index} follows.
\end{proof}

The second requirement concerns the output of the subroutine.

\begin{condition}[Subroutine admissibility] \label{cond:sub}
   For any $t \geq 0$, the following hold
    \begin{enumerate}[label=(\roman*)]
        \item \label{cond:d_t} The search direction satisfies $\|d_t\|_2 \leq D_{\cA}$ and the whole segment generated by this direction remains feasible, that is, $x_t-\gamma d_t \in \cX$ for any $\gamma \in [0, \gamma_t^{\max}]$.
        \item \label{cond:cG} The set of good iterations $\cG$ defined in \cref{eq:cG} is infinite.
        \item \label{cond:convex} If the function $f$ is convex, then there exists a constant $R\geq 1$, independent of $t$, such that 
        $$\grad f(x_t)^\top d_t \geq \frac{f(x_t)-f(x^\star)}{R},$$
        where $x^\star$ denotes an optimal solution of~\eqref{eq:opt}.
    \end{enumerate}
\end{condition}

The next lemma collects the basic properties of \cref{alg:ac-fw} that will be used repeatedly throughout the paper. 

\begin{lemma}
\label{lem:basic}
    Suppose that \cref{ass:X-compact,ass:smooth-functions} hold, and that $\{x_t\}_{t\geq 0}$ is the sequence generated by \cref{alg:ac-fw}. Under \cref{cond:r_t,cond:sub}, for every $t\geq 0$, the following hold.
    \begin{enumerate}[label=(\roman*)]
        \item \label{lem:basic:one-step}
        The trial point and the next iterate satisfy
        \begin{align*}
            f(x_{t+1})
            \leq
            f(\bar x_{t+1})
            \leq
            f(x_t)-\gamma_t \grad f(x_t)^\top d_t+\frac{L_{t+1}}{2}\gamma_t^2\|d_t\|_2^2.
        \end{align*}
        \item \label{lem:basic:I}
        If $t\in \cI_\eta$, then
        \begin{align*}
            f(x_{t+1})
            \leq
            f(x_t)-\left(1-\frac{\eta}{2}\right)\gamma_t \grad f(x_t)^\top d_t.
        \end{align*}

        \item \label{lem:basic:GI}
        If $t\in \cG\cap \cI_\eta$, then
        \begin{align*}
            f(x_{t+1})
            \leq
            f(x_t)-\left(1-\frac{\eta}{2}\right)
            \min\left\{
                1,
                \frac{\grad f(x_t)^\top d_t}{L_t\|d_t\|_2^2}
            \right\}\grad f(x_t)^\top d_t.
        \end{align*}
        \item \label{lem:basic:fw-gap}
        If, in addition, $f$ is convex and $\cX = \conv(\cA)$, then
        \begin{align*}
            \grad f(x_t)^\top (x_t-v_t)\geq f(x_t)-f(x^\star).
        \end{align*}
    \end{enumerate} 
\end{lemma}

\begin{proof}[Proof of \cref{lem:basic}]
    We first prove \ref{lem:basic:one-step}. By the update rule,
    \begin{align*}
        x_{t+1}
        =
        \begin{cases}
            \bar x_{t+1}, & \text{if } f(\bar x_{t+1})<f(x_t),\\
            x_t, & \text{otherwise}.
        \end{cases}
    \end{align*}
    Hence, in either case, $f(x_{t+1})\leq f(\bar x_{t+1})$.
    Moreover, by definition $L_{t+1}\geq \ell(x_t,\bar x_{t+1})$, implying
    \begin{align*}
        f(\bar x_{t+1})
        &\leq
        f(x_t)+\grad f(x_t)^\top(\bar x_{t+1}-x_t)
        +\frac{L_{t+1}}{2}\|\bar x_{t+1}-x_t\|_2^2.
    \end{align*}
    The claim then follows since, by construction, $\bar x_{t+1} - x_t = \gamma_t d_t$.

    We now prove \ref{lem:basic:I}. By the definition of $\gamma_t$, we have $\gamma_tL_t\|d_t\|_2^2\leq \grad f(x_t)^\top d_t$. Combining this with \ref{lem:basic:one-step}, we obtain
    \begin{align*}
        f(x_{t+1})
        &\leq
        f(x_t)-\left(1-\frac{L_{t+1}}{2L_t}\right)\gamma_t \grad f(x_t)^\top d_t.
    \end{align*}
    If $t\in \cI_\eta$, then $L_{t+1}\leq \eta L_t$. Substituting this bound into the inequality above establishes the claim.

    We next prove \ref{lem:basic:GI}. If $t\in\cG\cap \cI_\eta$, then either $\gamma_t^{\max}\geq 1$ or $\gamma_t<\gamma_t^{\max}$.
    If $\gamma_t^{\max}\geq 1$, then
    \begin{align*}
        \gamma_t
        =
        \min\left\{
            \frac{\grad f(x_t)^\top d_t}{L_t\|d_t\|_2^2},
            \gamma_t^{\max}
        \right\}
        \geq
        \min\left\{
            \frac{\grad f(x_t)^\top d_t}{L_t\|d_t\|_2^2},
            1
        \right\}.
    \end{align*}
    If instead $\gamma_t<\gamma_t^{\max}$, then necessarily
    \begin{align*}
        \gamma_t=\frac{\grad f(x_t)^\top d_t}{L_t\|d_t\|_2^2}
        \geq
        \min\left\{
            \frac{\grad f(x_t)^\top d_t}{L_t\|d_t\|_2^2},
            1
        \right\}.
    \end{align*}
    Hence, in both cases,
    \begin{align*}
        \gamma_t
        \geq
        \min\left\{
            1,
            \frac{\grad f(x_t)^\top d_t}{L_t\|d_t\|_2^2}
        \right\}.
    \end{align*}
    The claim then follows by combining this with \ref{lem:basic:I}.

    Finally, if $f$ is convex and $\cX=\conv(\cA)$, then since $v_t$ minimizes the linear functional $\grad f(x_t)^\top v$ over $\cA$ and $x^\star\in\conv(\cA)$, we have
    $\grad f(x_t)^\top(x_t-v_t) \geq \grad f(x_t)^\top(x_t-x^\star)$.
    By convexity of $f$, we additionally have $\grad f(x_t)^\top(x_t - x^\star) \geq f(x_t) - f(x^\star)$. Combining these two establishes the claim.
\end{proof}

We now verify that several classic Frank-Wolfe direction-finding rules satisfy these admissibility conditions. 
For each variant, we check \cref{cond:sub} and then combine this verification with \cref{lemma:index} to obtain a quantitative lower bound on the cardinality of the index set $\cG \cap \cI_\eta$, which will be used later in the convex case.

\subsection{Closed-Loop Frank-Wolfe Subroutine}
\label{sec:FW-subroutine}

We begin with the closed-loop Frank-Wolfe subroutine given in \cref{alg:fw-sub}, where the feasible set is defined as $\cX = \conv(\cA)$. 

\begin{lemma}\label{lem:fw-sub}
    Let $\eta > 1$. Under \cref{ass:X-compact,ass:smooth-functions}, the closed-loop Frank-Wolfe subroutine in \cref{alg:fw-sub} satisfies Conditions~\ref{cond:sub}\,\ref{cond:d_t} and \ref{cond:cG}. More precisely, for any $ t \geq 0$,
    \begin{align}
        \label{eq:fw:lb}
        |\cG \cap \cI_\eta \cap [t]| \geq t+1 -\left\lfloor\log_\eta\left(\frac{L}{r L_0}\right)\right\rfloor.
    \end{align}
    If additionally $f$ is convex, then \cref{cond:sub}\,\ref{cond:convex} holds with $R=1$.
\end{lemma}

\begin{proof}[Proof of \cref{lem:fw-sub}]
    Under \cref{alg:fw-sub}, we have $d_t = x_t-v_t$ and $\gamma_t^{\max}=1$. 
    Since $x_t\in \conv(\cA)$ and $v_t \in \cA \subseteq \conv(\cA)$, for any $\gamma \in[0,\gamma_t^{\max}]=[0,1]$, we have $x_t-\gamma d_t = (1-\gamma) x_t + \gamma v_t \in \conv(\cA).$
    Moreover, because both $x_t$ and $v_t$ belong to $\conv(\cA)$, and the diameter of $\conv(\cA)$ equals the diameter of $\cA$, it follows that $\|d_t\|_2 = \|x_t-v_t\|_2 \leq D_{\cA}.$
    Hence, \cref{cond:sub}\,\ref{cond:d_t} is satisfied.
    
    Next, since $\gamma_t^{\max}=1$ for every $t\geq 0$, every iteration belongs to $\cG$ by
    definition. Therefore, $\cG=\bbZ_+$, and in particular \cref{cond:sub}\,\ref{cond:cG} holds.
    We now prove the cardinality bound. 
    Since $\cG = \bbZ_+$, we have $|\cG \cap \cI_\eta \cap [t]| = |\cI_\eta  \cap [t]|$. The bound~\eqref{eq:fw:lb} now follows directly from \cref{lemma:index}, which implies that among the first $t+1$ iterations, at most $\left\lfloor \log_\eta\!\left(\frac{L}{rL_0}\right)\right\rfloor$ indices can lie outside $\cI_\eta$. 
    
    Finally, suppose that $f$ is convex. Since here $\cX=\conv(\cA)$, \cref{lem:basic}\,\ref{lem:basic:fw-gap} yields $\grad f(x_t)^\top d_t = \grad f(x_t)^\top(x_t-v_t) \geq
    f(x_t)-f(x^\star).$ Thus, \cref{cond:sub}\,\ref{cond:convex} holds with $R=1$.
\end{proof}

\subsection{Matching Pursuit Subroutine}
\label{sec:MP-subroutine}
We next consider the Matching Pursuit subroutine given in \cref{alg:mp-sub}, where the feasible set is defined as $\cX = \lin(\cA)$. 
We verify that this choice fits our abstract framework and then deduce a lower bound on the cardinality of $\cG \cap \cI_\eta$.
To do so, recall that given a set $\cA \subset \R^n$ such that $0 \in \conv(\cA)$, the associated atomic norm of a vector $x \in \R^n$ is defined by $\| x \|_{\cA} := \inf\{c > 0 : x \in c \cdot \conv(\cA) \}$.

\begin{assumption}[Regular dictionary] \label{ass:mp}
    The dictionary $\cA $ is symmetric around the origin, that is, $\cA = -\cA$. Furthermore, if $f$ is convex, the initial level set has bounded atomic radius, that is,
    \begin{align*}
        R_{\cA} = \max\left\{ 1,\; \sup\left\{ \|x-x^\star\|_{\cA}: x \in \lin(\cA),\, f(x)\leq f(x_0) \right\} \right\} <\infty.
    \end{align*}
\end{assumption}

The symmetry part of \cref{ass:mp} is without loss of generality for Matching Pursuit. Indeed, replacing $\cA$ by its symmetrization $\tilde{\cA}:= \cA \cup(-\cA)$ does not change the feasible set, since $\lin(\tilde{\cA}) = \lin(\cA)$.
By contrast, the finiteness of $R_{\cA}$ is a regularity assumption on the initial level set. Nonetheless, this condition is only required to show that \cref{cond:sub}\,\ref{cond:convex} holds.

\begin{lemma}\label{lem:mp-sub}
    Let $\eta > 1$. Under \cref{ass:X-compact,ass:smooth-functions,ass:mp}, the Matching Pursuit subroutine in \cref{alg:mp-sub} satisfies Conditions~\ref{cond:sub}\,\ref{cond:d_t} and \ref{cond:cG}.
    More precisely, for any $ t \geq 0$, the bound \eqref{eq:fw:lb} holds for the MP subroutime as well. If additionally $f$ is convex, then \cref{cond:sub}\,\ref{cond:convex} holds with $R = R_{\cA}$.
\end{lemma}

\begin{proof}[Proof of \cref{lem:mp-sub}]
    Under \cref{alg:mp-sub}, we have $d_t = -v_t$ and $\gamma_t^{\max} = \infty.$
    Since $\cA = -\cA$, $-v_t \in \cA$, and therefore $2\|v_t\|_2 = \| v_t - (-v_t) \|_2 \leq D_{\cA}$.
    Thus, we have $\|d_t\|_2 = \|v_t\|_2 \leq D_{\cA}.$
    Furthermore, for any $\gamma \in [0, \infty)$, $x_t - \gamma d_t = x_t + \gamma v_t \in \lin(\cA) = \cX$, since both $x_t$ and $v_t$ belong to $\lin(\cA)$. Thus, \cref{cond:sub}\,\ref{cond:d_t} holds.

    Next, since $\gamma_t^{\max} = \infty \geq 1$ for every $t \geq 0$, every iteration belongs to $\cG$. Therefore, $\cG = \bbZ_+$, and in particular \cref{cond:sub}\,\ref{cond:cG} holds. The bound~\eqref{eq:fw:lb} now follows directly from \cref{lemma:index}.

    Finally, suppose that $f$ is convex. Since the acceptance rule in \cref{alg:ac-fw} guarantees $f(x_t)\leq f(x_0)$ for all $t\geq 0$, \cref{ass:mp} implies $\|x_t - x^\star\|_{\cA} \leq R_{\cA}$.
    We thus have
    \begin{align*}
        f(x_t) - f(x^\star) 
        \leq \grad f(x_t)^\top (x_t-x^\star)
        \leq
        \Big(\max_{a\in \cA} \grad f(x_t)^\top a \Big) \cdot \|x_t-x^\star\|_{\cA}
        \leq
        R_{\cA} \Big(\max_{a\in \cA} \grad f(x_t)^\top a \Big),
    \end{align*}
    where the first inequality follows from convexity of $f$ and the second inequality follows from the polar inequality \citep[Theorem~15.1]{rockafellar1997convex}.
    Since $\cA = -\cA$ and $v_t \in \argmin_{v \in \cA} \grad f(x_t)^\top v$, we have
    \begin{align*}
        \max_{a\in \cA} \grad f(x_t)^\top a
        = -\min_{a\in \cA} \grad f(x_t)^\top a
        = -\grad f(x_t)^\top v_t
        = \grad f(x_t)^\top d_t.
    \end{align*}
    Combining the above inequalities yields $f(x_t)-f(x^\star)\leq R_{\cA}\,\grad f(x_t)^\top d_t.$
    Thus, \cref{cond:sub}\,\ref{cond:convex} holds with $R=R_{\cA}$.
\end{proof}

\subsection{Pairwise Frank-Wolfe Subroutine}
\label{sec:PFW-subroutine}

We next consider the pairwise Frank-Wolfe subroutine given in \cref{alg:pairwise-sub}, where the feasible set is $\cX=\conv(\cA)$.
This variant is naturally implemented by maintaining an active-set representation $x_t = \sum_{s\in\cA} \alpha_{s,t} s$, where, by convention, $\alpha_{s,t}=0$ for $s\notin\cS_t$. The essential idea of \cref{alg:pairwise-sub} is to ``swap out'' the worst performing atom of the current active set $s_t \in \cS_t$, and swap in the most promising atom $v_t \in \cA$.
This can be implemented recursively by updating the pair $(\cS_t, \alpha_{\cdot,t})$ from one iteration to the next. If the trial point is accepted, i.e., $x_{t+1} = \bar x_{t+1}$, then the coefficients are updated according to
\begin{align*}
    \alpha_{s,t+1}
    = \alpha_{s,t}
    - \gamma_t\,\mathds{1}_{\{s=s_t\}}
    +\gamma_t\,\mathds{1}_{\{s=v_t\}},
    \qquad \forall s\in\cA.
\end{align*}
If the trial point is rejected, i.e., $x_{t+1}=x_t$, then
\begin{align*}
    \alpha_{s,t+1}=\alpha_{s,t},
    \qquad \forall s\in\cA.
\end{align*}
In all cases, the next active set is recovered via $\cS_{t+1} = \{s \in \cA : \alpha_{s,t+1} > 0 \}$. We now verify that this choice fits our abstract framework and then derive a lower bound on the cardinality of $\cG\cap\cI_\eta$ under the following additional assumption.

\begin{assumption}[Finite dictionary]\label{ass:atom}
    The set of atoms $\cA$ is finite.
\end{assumption}

\begin{lemma}\label{lem:pairwise-sub}
    Let $\eta \in (1, 2)$. Under \cref{ass:X-compact,ass:smooth-functions,ass:atom}, the pairwise Frank-Wolfe subroutine in \cref{alg:pairwise-sub} satisfies Conditions~\ref{cond:sub}\,\ref{cond:d_t} and \ref{cond:cG}.
    More precisely, for any $ t \geq 0$, we have
    \begin{align}
        \label{eq:pfw:lb}
        |\cG\cap \cI_\eta \cap [t]|
        \geq
        \frac{t+1}{3|\cA|!+1}
        - \left\lfloor \log_\eta \left(\frac{L}{rL_0}\right) \right\rfloor.
    \end{align}
    If $f$ is additionally convex, then \cref{cond:sub}\,\ref{cond:convex} holds with $R=1$.
\end{lemma}

\begin{proof}[Proof of \cref{lem:pairwise-sub}]
    Under \cref{alg:pairwise-sub}, we have $d_t = s_t - v_t$ and $\gamma_t^{\max} = \alpha_{s_t,t}$. Since $s_t, v_t \in \cA$, we have $\|d_t\|_2=\|s_t-v_t\|_2\leq D_{\cA}$.
    Moreover, for any $\gamma \in [0,\gamma_t^{\max}] = [0,\alpha_{s_t,t}]$, writing $x_t = \sum_{s \in \cS_t} \alpha_{s,t}s$,
    we obtain $x_t - \gamma d_t = \sum_{s \in \cS_t \setminus \{s_t,v_t\}} \alpha_{s,t} s + (\alpha_{s_t,t}-\gamma) s_t +(\alpha_{v_t,t} + \gamma) v_t$, where we use the convention $\alpha_{v_t,t} = 0$ if $v_t \notin \cS_t$. All coefficients remain nonnegative and still sum to $1$, and therefore $x_t - \gamma d_t \in \conv(\cA) = \cX.$
    Hence, \cref{cond:sub}\,\ref{cond:d_t} holds.

    We next prove the lower bound in \eqref{eq:pfw:lb}, from which \cref{cond:sub}\,\ref{cond:cG} will follow. 
    Define the set
    \begin{align*}
        \cH := \{k \geq 0 : f(\bar x_{k+1})<f(x_k) \}.
    \end{align*}
    Let $\ell_0 < \ell_1 < \cdots < \ell_m$ be the elements of \(\cH\cap[t]\). We first aim to show
    \begin{equation}
        \label{eq:pairwise-count}
        |\cH\cap\cG\cap[t]|
        \geq
        \frac{|\cH\cap[t]|}{3|\cA|!+1}.
    \end{equation}
    Define $J_t:=\{ j \in [m] : \ell_j \notin \cG \}$.
    If $J_t = \emptyset$, then $|\cH \cap \cG \cap[t]| = |\cH \cap [t]|$, so \eqref{eq:pairwise-count} is immediate. Hence we may assume $J_t \neq \emptyset$.
    Write $J_t = \bigcup_{r \in [R_t]} \{a_r,a_r+1,\dots,b_r\}$ as union of disjoint integer intervals, 
    where $a_0\leq b_0<a_1\leq b_1<\cdots<a_{R_t}\leq b_{R_t}$.
    We claim
    \begin{align*}
        b_r - a_r + 1 \leq 3|\cA|!,
        \qquad r \in [R_t].
    \end{align*}
    Indeed, for every $j \in \{a_r,\dots,b_r\}$, since $\ell_j \notin \cG$, we have $\gamma_{\ell_j}^{\max} < 1$ and $\gamma_{\ell_j}=\gamma_{\ell_j}^{\max}=\alpha_{s_{\ell_j},\ell_j}$.
    Thus, we have $|\cS_{\ell_j}| \geq 2$, and the atom $s_{\ell_j}$ is removed from the active set at the next iterate. 
    If $v_{\ell_j} \in \cS_{\ell_j}$, then the active-set size decreases by $1$. 
    If $v_{\ell_j} \notin \cS_{\ell_j}$, then the active-set size stays the same as $\cS_{\ell_j+1} = (\cS_{\ell_j} \setminus \{s_{\ell_j}\}) \cup \{v_{\ell_j}\}.$
    Since \(\ell_j\in\cH\), we also have $f(x_{\ell_j+1}) < f(x_{\ell_j}),$ so the same iterate cannot reappear.

    Fix $c\in\{2,\dots,|\cA|\}$. Whenever $|\cS_{\ell_j}| = c$, the number of possible iterates is at most ${|\cA|!}/{(|\cA|-c)!}$.
    Therefore, we have
    \begin{align*}
        b_r - a_r + 1 \leq \sum_{b=2}^{|\cA|}\frac{|\cA|!}{(|\cA|-c)!} = |\cA|!\sum_{q=0}^{|\cA|-2}\frac{1}{q!} < e|\cA|! < 3|\cA|!,
    \end{align*}
    as desired.

    We next find a lower bound for $|\cH \cap \cG \cap [t]|$.
    Since $\cS_0 = \{x_0\}$, the first element of $\cH \cap [t]$ belongs to~$\cG$. Hence $a_0 \neq 0$, so $a_r-1 \in [m] \setminus J_t$ for every $r$, and these indices are distinct. Thus, $R_t + 1 \leq |[m] \setminus J_t|$.
    Consequently, we have $|J_t| = \sum_{r \in [R_t]} (b_r-a_r+1) \leq 3|\cA|!\,(R_t+1) \leq 3|\cA|!\,|[m] \setminus J_t|.$
    Since
    \begin{align*}
        |[m] \setminus J_t| = |\cH \cap \cG \cap [t]|
        \qquad\text{and}\qquad
        |J_t| + |[m] \setminus J_t| = m+1 = |\cH \cap [t]|,
    \end{align*}
    we conclude that \cref{eq:pairwise-count} holds.
    
    We now show that $\cI_\eta \cap \cG^c \cap [t] \subseteq \cH \cap \cG^c \cap [t]$. Let $k \in \cI_\eta \cap \cG^c \cap[t]$. Since $k \notin \cG$, we have $\gamma_k^{\max} < 1$ and $\gamma_k = \gamma_k^{\max}$.
    If $k \notin \cH$, then $f(\bar x_{k+1}) \geq f(x_k)=f(x_{k+1})$. By \cref{lem:basic}\,\ref{lem:basic:one-step}, we have
    \begin{align*}
        0
        \leq f(\bar x_{k+1})-f(x_k)
        \leq -\gamma_k \grad f(x_k)^\top d_k+\frac{L_{k+1}}{2}\gamma_k^2\|d_k\|_2^2.
    \end{align*}
    Moreover, 
    \begin{align*}
        \gamma_k 
        = \min\left\{
            \frac{\grad f(x_k)^\top d_k}{L_k\|d_k\|_2^2},
            \gamma_k^{\max}
        \right\}
        = \gamma_k^{\max} 
        \implies
        \gamma_kL_k\|d_k\|_2^2\leq \grad f(x_k)^\top d_k,
    \end{align*}
    and therefore $0 \leq f(\bar x_{k+1})-f(x_k) \leq (\tfrac{L_{k+1}}{2}-L_k) \gamma_k^2\|d_k\|_2^2.$
    Hence $L_{k+1}\geq 2L_k$, which contradicts $k\in\cI_\eta$, since $\eta<2$. Therefore $k\in\cH$.

    It thus follows that $|\cI_\eta \cap \cG^c \cap [t]| \leq |\cH \cap \cG^c \cap [t]|$.
    By \eqref{eq:pairwise-count}, we have
    \begin{align*}
        |\cH \cap \cG^c \cap [t]| = |\cH \cap [t]| - |\cH \cap \cG \cap [t]| \leq
        3|\cA|!\, |\cH \cap \cG \cap [t]|.
    \end{align*}
    Since $\cH \cap \cG \cap [t] \subseteq (\cG \cap \cI_\eta \cap [t]) \cup (\cI_\eta^c \cap [t]),$
    we obtain $|\cI_\eta \cap \cG^c \cap [t]| \leq 3|\cA|! \, (|\cG \cap \cI_\eta \cap [t]| + |\cI_\eta^c \cap [t]|)$.

    Using the first part of \cref{lemma:index},
    \begin{align*}
        |\cI_\eta \cap [t]|
        \geq t+1 -\left\lfloor \log_\eta \left(\frac{L}{rL_0}\right) \right\rfloor
        \qquad \text{and} \qquad
        |\cI_\eta^c \cap [t]|
        \leq \left\lfloor \log_\eta \left(\frac{L}{rL_0}\right) \right\rfloor.
    \end{align*}
    Since $|\cG\cap\cI_\eta\cap[t]| = |\cI_\eta\cap[t]| - |\cI_\eta\cap\cG^c\cap[t]|,$ we have
    \[
        |\cG\cap\cI_\eta\cap[t]|
        \geq
        t+1-\left\lfloor\log_\eta\left(\frac{L}{rL_0}\right)\right\rfloor
        -3|\cA|!\Bigl(|\cG\cap\cI_\eta\cap[t]|+\left\lfloor\log_\eta\left(\frac{L}{rL_0}\right)\right\rfloor\Bigr).
    \]
    Rearranging yields \eqref{eq:pfw:lb}. Since the right-hand side of \eqref{eq:pfw:lb} tends to $+\infty$ as $t \to \infty$, it follows that the set $\cG \cap \cI_\eta$ is infinite, and therefore $\cG$ is infinite as well. This proves \cref{cond:sub}\,\ref{cond:cG}.

    Finally, suppose that $f$ is convex. Since $x_t$ is a convex combination of atoms in $\cS_t$ and $s_t$ maximizes the linear functional $\grad f(x_t)^\top s$ over $\cS_t$, we have $\grad f(x_t)^\top(s_t-x_t)\geq 0$.
    Using \cref{lem:basic}\,\ref{lem:basic:fw-gap}, we obtain
    \begin{align*}
        \grad f(x_t)^\top d_t
        = \grad f(x_t)^\top(s_t-x_t) + \grad f(x_t)^\top(x_t-v_t)
        \geq \grad f(x_t)^\top(x_t-v_t)
        \geq f(x_t)-f(x^\star).
    \end{align*}
    Hence, \cref{cond:sub}\,\ref{cond:convex} holds with $R=1$.
\end{proof}

\subsection{Away-Step Frank-Wolfe Subroutine}
\label{sec:AFW-subroutine}

We now consider the away-step Frank-Wolfe subroutine given in \cref{alg:away-sub}, where the feasible set is defined as $\cX = \conv(\cA)$. 
The away-step variant can be implemented recursively by carrying the pair $(\cS_t, \alpha_{\cdot,t})$ from one iteration to the next. 
If the trial point is accepted, i.e., $x_{t+1}=\bar x_{t+1}$, then an accepted Frank-Wolfe step gives
\begin{align*}
    \alpha_{s,t+1}
    = (1-\gamma_t) \alpha_{s,t}
    +\gamma_t\,\mathds{1}_{\{s=v_t\}},
    \qquad \forall s\in \cA,
\end{align*}
whereas an accepted away step gives
\begin{align*}
    \alpha_{s,t+1}
    =
    (1+\gamma_t)\alpha_{s,t}
    -\gamma_t\,\mathds{1}_{\{s=s_t\}},
    \qquad \forall s\in \cA.
\end{align*}
If the trial point is rejected, i.e., $x_{t+1}=x_t$, then
\begin{align*}
    \alpha_{s,t+1}=\alpha_{s,t},
    \qquad \forall s\in \cA.
\end{align*}
In all cases, the next active set is recovered from $\cS_{t+1} = \{s\in \cA : \alpha_{s,t+1} > 0 \}$.
We now verify that this choice fits our abstract framework and then deduce a lower bound on the cardinality of $\cG \cap \cI_\eta$.

\begin{lemma}\label{lem:away-sub}
    Let $\eta\in(1,2)$. Under \cref{ass:X-compact,ass:smooth-functions}, the away-step Frank-Wolfe subroutine in \cref{alg:away-sub} satisfies \cref{cond:sub}\,\ref{cond:d_t} and \cref{cond:sub}\,\ref{cond:cG}. More precisely, for any $t\geq 0$,
    \begin{align}
        \label{eq:afw:lb}
        |\cG\cap \cI_\eta \cap [t]|
        \geq
        \frac{t+1}{2}
        -\left\lfloor\log_\eta\left(\frac{L}{rL_0}\right)\right\rfloor.
    \end{align}
    If $f$ is additionally convex, then \cref{cond:sub}\,\ref{cond:convex} holds with $R=1$.
\end{lemma}

\begin{proof}[Proof of \cref{lem:away-sub}]
    At every iteration, the away-step subroutine chooses either the Frank-Wolfe direction $(d_t, \gamma_t^{\max}) = (x_t-v_t, 1)$ or the away direction $(d_t, \gamma_t^{\max}) = (s_t - x_t, {\alpha_{s_t,t}}/{(1-\alpha_{s_t,t}}))$.
    If the Frank-Wolfe direction is chosen, then $\|d_t\|_2=\|x_t-v_t\|_2\leq D_{\cA}$, since $x_t, v_t \in \conv(\cA)$. Moreover, for every $\gamma \in [0,1]$, $x_t - \gamma d_t = (1-\gamma)x_t + \gamma v_t \in \conv(\cA) = \cX$.
    If the away direction is chosen, then $\|d_t\|_2=\|s_t-x_t\|_2\leq D_{\cA}$, since again $s_t, x_t \in \conv(\cA)$. Writing $x_t = \sum_{s \in \cS_t} \alpha_{s,t} s$, we have, for every $\gamma\in[0,\gamma_t^{\max}]$, $x_t -\gamma d_t = (1 + \gamma) x_t - \gamma s_t = \sum_{s\in\cS_t\setminus\{s_t\}}(1+\gamma) \alpha_{s,t} s + ((1+\gamma) \alpha_{s_t,t} - \gamma) s_t$.
    All coefficients are nonnegative, because
    \begin{align*}
        (1+\gamma)\alpha_{s_t,t}-\gamma \geq 0
        \qquad \Longleftrightarrow \qquad 
        \gamma \leq \frac{\alpha_{s_t,t}}{1-\alpha_{s_t,t}}=\gamma_t^{\max},
    \end{align*}
    and they sum to $(1+\gamma)\sum_{s\in\cS_t}\alpha_{s,t} -\gamma = (1+\gamma) - \gamma = 1.$ Hence, $x_t-\gamma d_t\in\conv(\cA)=\cX$. This proves \cref{cond:sub}\,\ref{cond:d_t}.

    We next prove the lower bound in \eqref{eq:afw:lb}, from which \cref{cond:sub}\,\ref{cond:cG} will follow. 
    We first claim that every index in $\cG^c\cap \cI_\eta$ corresponds to an accepted away step that removes one atom from the active set. Let $k \in \cG^c\cap \cI_\eta$. Since $k \notin \cG$, by definition we have $\gamma_k^{\max} < 1$ and $\gamma_k = \gamma_k^{\max}.$
    In particular, the Frank-Wolfe direction cannot have been chosen, because in that case $\gamma_k^{\max}=1$. Hence, the away direction is chosen at iteration $k$, so $d_k=s_k-x_k$ and $\gamma_k^{\max}=\frac{\alpha_{s_k,k}}{1-\alpha_{s_k,k}}<1.$

    Suppose, for contradiction, that the trial point is rejected, i.e., $x_{k+1}=x_k$. Then $f(\bar x_{k+1}) \geq f(x_k)=f(x_{k+1})$. By \cref{lem:basic}\,\ref{lem:basic:one-step}, we have
    \begin{align*}
        0
        \leq f(\bar x_{k+1})-f(x_k)
        \leq -\gamma_k \grad f(x_k)^\top d_k+\frac{L_{k+1}}{2}\gamma_k^2\|d_k\|_2^2.
    \end{align*}
    Moreover, 
    \begin{align*}
        \gamma_k 
        = \min\left\{
            \frac{\grad f(x_k)^\top d_k}{L_k\|d_k\|_2^2},
            \gamma_k^{\max}
        \right\}
        = \gamma_k^{\max} 
        \implies
        \gamma_kL_k\|d_k\|_2^2\leq \grad f(x_k)^\top d_k,
    \end{align*}
    and therefore $0 \leq f(\bar x_{k+1})-f(x_k) \leq (\tfrac{L_{k+1}}{2}-L_k) \gamma_k^2\|d_k\|_2^2.$
    Hence $L_{k+1}\geq 2L_k$, which contradicts $k\in\cI_\eta$, since $\eta<2$. Thus, the step at iteration $k$ must be accepted.
    Since it is an accepted away step with $\gamma_k = \gamma_k^{\max}$, the atom $s_k$ is removed from the active set, so the active-set size decreases by one. This proves the claim.
    
    Next, let $A_t := |\{ k \in [t]: |\cS_{k+1}| > |\cS_k| \}|$ be the number of iterations up to time $t$ at which the active-set size increases and $D_t := |\{k \in [t] : |\cS_{k+1}| < |\cS_k|\}|$ be the number of iterations up to time $t$ at which the active-set size decreases.
    Since $|\cS_0|=1$ and the active set is always nonempty, we obtain $1 \leq |\cS_{t+1}| = |\cS_0| + A_t - D_t = 1 + A_t - D_t$. Furthermore, as shown above $D_t \geq |\cG^c \cap \cI_\eta \cap [t]|$, so $1 \leq 1 + A_t - |\cG^c \cap \cI_\eta \cap [t]|$ hence $|\cG^c \cap \cI_\eta \cap [t]| \leq A_t$. Since $A_t + |\cG^c \cap \cI_\eta \cap [t]| \leq t+1$, we thus have
    \[ |\cG^c \cap \cI_\eta \cap [t]| \leq \frac{t+1}{2}. \]
    On the other hand, by \cref{lemma:index},
    \begin{align*}
        |\cI_\eta^c\cap [t]|
        \leq C:= \left\lfloor\log_\eta\left(\frac{L}{rL_0}\right)\right\rfloor
        \implies 
        |(\cG\cap \cI_\eta)^c\cap [t]|
        \leq
        |\cG^c\cap \cI_\eta\cap [t]|+|\cI_\eta^c\cap [t]|
        \leq
        \frac{t+1}{2}+C.
    \end{align*}
    Using $|\cG\cap \cI_\eta\cap [t]| = t+1-|(\cG\cap \cI_\eta)^c\cap [t]|$, we arrive at~\eqref{eq:afw:lb}. Since the right-hand side tends to $+\infty$ as $t \to \infty$, it follows that the set $\cG \cap \cI_\eta \cap [t]$ is infinite, and therefore $\cG$ is infinite as well. This proves \cref{cond:sub}\,\ref{cond:cG}.
    
    Finally, suppose that $f$ is convex. If the away direction is chosen, then by the selection rule and \cref{lem:basic}\,\ref{lem:basic:fw-gap}, we have $\grad f(x_t)^\top(s_t-x_t)  > \grad f(x_t)^\top(x_t-v_t)  \geq f(x_t)-f(x^\star)$.
    Similarly, if the Frank-Wolfe step is chosen, we have $\grad f(x_t)^\top(x_t-v_t) \geq f(x_t)-f(x^\star)$ thanks to \cref{lem:basic}\,\ref{lem:basic:fw-gap}.
    Therefore, \cref{cond:sub}\,\ref{cond:convex} holds with $R=1$.
\end{proof}

\section{General Convergence Guarantees}
\label{sec:convergence}

This section develops the convergence theory that is common to all admissible subroutines in our framework. More precisely, throughout this section we only assume \cref{ass:X-compact,ass:smooth-functions} together with Conditions~\ref{cond:r_t} and~\ref{cond:sub}, and we do not impose any additional structure. We first derive a generic stationarity guarantee that applies in the nonconvex setting, and then show how the same framework yields objective-value convergence under convexity. In this sense, the results of this section form the baseline convergence theory for the auto-conditioned scheme introduced in \cref{sec:framework}.

\subsection{Nonconvex settings}
In the nonconvex setting, the quantity $\grad f(x_t)^\top d_t$ is a natural stationarity measure for our framework. Indeed, when $\cX=\conv(\cA)$, a first-order stationary point $x\in\cX$ satisfies
\begin{align*}
    \max_{v\in\cA} \grad f(x)^\top(x-v)=0.
\end{align*}
This is exactly the first-order optimality condition when the problem is convex. For the closed-loop Frank-Wolfe, away-step, and pairwise subroutines, one always has
\begin{align*}
    \grad f(x_t)^\top d_t
    \geq
    \grad f(x_t)^\top(x_t-v_t)
    =
    \max_{v\in\cA}\grad f(x_t)^\top(x_t-v).
\end{align*}
Thus, controlling $\grad f(x_t)^\top d_t$ controls a valid first-order stationarity measure over $\conv(\cA)$.

When $\cX=\lin(\cA)$, the corresponding first-order condition is
\begin{align*}
    \max_{v\in\cA}\grad f(x)^\top v = 0.
\end{align*}
Under the Matching Pursuit subroutine, we have $d_t=-v_t$, and thus, $\grad f(x_t)^\top d_t = -\min_{v\in\cA}\grad f(x_t)^\top v$.
Under the symmetry assumption $\cA=-\cA$, we have
\begin{align*}
    \grad f(x_t)^\top d_t
    =
    \max_{v\in\cA}\grad f(x_t)^\top v.
\end{align*}
Therefore, in both regimes, the quantity $\grad f(x_t)^\top d_t$ provides a meaningful measure of first-order nonstationarity.

\begin{theorem} \label{thm:nonconvex}
    Suppose $\{x_t\}_{t \geq 0}$ is the sequence generated by \cref{alg:ac-fw}, and that \cref{ass:X-compact,ass:smooth-functions,cond:r_t,cond:sub} all hold. Let $\eta\in(1,2)$. For any $t \geq 0$ such that $|\cG \cap \cI_\eta \cap [t] | \geq 1$, there exists $k \in \cG \cap \cI_\eta \cap [t]$ such that
    \begin{align*}
        \grad f(x_{k})^\top d_k
        \leq
        \begin{cases}
            \sqrt{\dfrac{2LD_{\cA}^2\bigl(f(x_0)-f(x^\star)\bigr)}{(2-\eta) \cdot |\cG \cap \cI_\eta \cap [t] | }},
            & \grad f(x_k)^\top d_k \leq LD_{\cA}^2,\\[2ex]
            \dfrac{2\bigl(f(x_0)-f(x^\star)\bigr)}{(2-\eta) \cdot |\cG \cap \cI_\eta \cap [t] |},
            & \grad f(x_k)^\top d_k > LD_{\cA}^2.
        \end{cases}
    \end{align*}
\end{theorem}

\begin{proof}[Proof of \cref{thm:nonconvex}]
    For every $k \in \cG \cap \cI_\eta \cap [t]$, \cref{lem:basic}\,\ref{lem:basic:GI} gives
    \begin{align*}
        f(x_{k+1})
        \leq
        f(x_k)-\left(1-\frac{\eta}{2}\right)
        \min\left\{
            1,
            \frac{\grad f(x_k)^\top d_k}{L_k\|d_k\|_2^2}
        \right\}\grad f(x_k)^\top d_k.
    \end{align*}
    Summing this inequality over all $k \in \cG \cap \cI_\eta \cap [t]$, we obtain
    \begin{align*}
        \left(1-\frac{\eta}{2}\right)
        \sum_{k \in \cG \cap \cI_\eta \cap [t]}
        \min\left\{
            1,
            \frac{\grad f(x_k)^\top d_k}{L_k\|d_k\|_2^2}
        \right\}\grad f(x_k)^\top d_k
        &\leq \sum_{k \in \cG \cap \cI_\eta \cap [t]} \bigl( f(x_k)-f(x_{k+1}) \bigr) \\
        &\leq
        \sum_{k \in [t]} \bigl( f(x_k) - f(x_{k+1}) \bigr) \\
        &=
        f(x_0) - f(x_{t+1}) \\
        &\leq
        f(x_0) - f(x^\star).
    \end{align*}
    Using \cref{lemma:Lipschitz} and \cref{cond:sub}\,\ref{cond:d_t}, we deduce that
    \begin{align*}
        \left(1-\frac{\eta}{2}\right)
        \sum_{k \in \cG \cap \cI_\eta \cap [t]}
        \min\left\{
            1,
            \frac{\grad f(x_k)^\top d_k}{LD_{\cA}^2}
        \right\}\grad f(x_k)^\top d_k
        \leq
        f(x_0)-f(x^\star).
    \end{align*}
    Hence there exists $k \in \cG \cap \cI_\eta \cap [t]$ such that
    \begin{align*}
        \min\left\{
            1,
            \frac{\grad f(x_k)^\top d_k}{LD_{\cA}^2}
        \right\}\grad f(x_k)^\top d_k
        \leq
        \frac{2\bigl(f(x_0)-f(x^\star)\bigr)}{(2-\eta) \cdot |\cG \cap \cI_\eta \cap [t]|}.
    \end{align*}
    This yields the desired result, concluding the proof.
\end{proof}

For the subroutines studied in \cref{alg:fw-sub,alg:mp-sub,alg:pairwise-sub,alg:away-sub}, the lower bounds \cref{lem:fw-sub,lem:away-sub,lem:mp-sub,lem:pairwise-sub} imply that $|\cG \cap \cI_\eta \cap[t]|$ is of order $t$. As a simple consequence of this and \cref{thm:nonconvex}, we obtain the following guarantee.

\begin{corollary} \label{cor:nonconvex}
    Suppose $\{x_t\}_{t \geq 0}$ is the sequence generated by \cref{alg:ac-fw}, and that \cref{ass:X-compact,ass:smooth-functions,cond:r_t,cond:sub} all hold. Let $\eta \in (1,2)$, and let $h_0 := \min\{ t \geq 0 : t \in \cG \cap \cI_\eta \}$ denote the first iteration that achieves significant descent. Then for every $t \geq h_0$,
    \begin{align*}
        \min_{k \in [t]} \grad f(x_k)^\top d_k \leq O(1/\sqrt{t}).
    \end{align*}
\end{corollary}

\subsection{Convex settings}\label{sec:convex-convergence}

Throughout this subsection, we impose the following convexity assumption.

\begin{assumption}[Convex objective]
\label{ass:convex-functions}
   The objective function $f: \R^n \to \R$ is convex.
\end{assumption}

Under \cref{ass:convex-functions}, the quantity $\grad f(x_t)^\top d_t$ serves not only as a stationarity measure, but also as an upper bound on the objective error through \cref{cond:sub}\,\ref{cond:convex}. 
In this regime, it is convenient to first establish convergence estimates along the controlled subsequence indexed by $\cG \cap \cI_\eta$, and then translate these estimates to the full sequence using lower bounds on the cardinality of $\cG \cap \cI_\eta \cap[t]$. 
Accordingly, let $\{h_t\}_{t \geq 0}$ denote the elements of $\cG \cap \cI_\eta$ in increasing order.

\begin{lemma}
\label{lem:conv:translate}
    For every $t \geq 0$ and every integer $q \geq 0$, if $|\cG \cap \cI_\eta \cap[t]| \geq q+1$,
    then $h_q \leq t$.
    Consequently,
    \begin{align*}
        f(x_t) - f(x^\star) \leq f(x_{h_q}) - f(x^\star).
    \end{align*}
\end{lemma}

\begin{proof}[Proof of \cref{lem:conv:translate}]
    By definition, $h_q$ is the $(q+1)$-st smallest element of the set $\cG\cap\cI_\eta$. If $|\cG \cap \cI_\eta \cap [t]| \geq q+1,$ then at least $q+1$ elements of $\cG\cap\cI_\eta$ lie in $[t]$, and therefore $h_q \leq t$.
    Since the sequence $\{f(x_t)\}_{t \geq 0}$ is nonincreasing, it follows that $f(x_t) \leq f(x_{h_q})$, and hence $f(x_t) - f(x^\star) \leq f(x_{h_q}) - f(x^\star)$.
\end{proof}
The next theorem shows that this yields the classical sublinear $\cO(1/t)$ convergence rate for $\{x_{h_t}\}_{t\geq 0}$, without requiring prior knowledge of the global Lipschitz constant. 

\begin{theorem}
\label{thm:convex}
    Suppose $\{x_t\}_{t \geq 0}$ is the sequence generated by \cref{alg:ac-fw}, let $\eta \in (1,2)$, and assume \cref{ass:X-compact,ass:convex-functions,ass:smooth-functions} together with Conditions~\ref{cond:r_t} and \ref{cond:sub}. Then, for any $t \geq 0$, it holds that
    \begin{align*}
        f(x_{h_t})-f(x^\star)
        \leq
        \frac{\max\left\{\left(\frac{2R}{2-\eta}-1\right)\bigl(f(x_0)-f(x^\star)\bigr),\frac{2LR^2D_{\cA}^2}{2-\eta}\right\}}{t+\frac{2R}{2-\eta}-1},
    \end{align*}
    where $\{h_t\}_{t \geq 0}$ are the indices in $\cG \cap \cI_\eta$ in increasing order.
\end{theorem}

\begin{proof}[Proof of \cref{thm:convex}]
    Define
    \begin{align*}
        \Delta_t := f(x_{h_t}) - f(x^\star),
        \qquad
        \alpha := \frac{2-\eta}{2R},
        \qquad
        \beta := \frac{2-\eta}{2R^2 L D_{\cA}^2},
        \qquad
        \delta := \frac{1}{\alpha} - 1 = \frac{2R}{2-\eta} - 1.
    \end{align*}
    Let
    \begin{align*}
        C
        :=
        \max\left\{
            \delta \left(f(x_0)-f(x^\star)\right),
            \frac{1}{\beta}
        \right\}
        =
        \max\left\{
            \left(\frac{2R}{2-\eta}-1\right)\bigl(f(x_0)-f(x^\star)\bigr),
            \frac{2LR^2D_{\cA}^2}{2-\eta}
        \right\}.
    \end{align*}
    We first derive a recursion for $\Delta_t$. Let $k:=h_t$. Since $k\in \cG\cap\cI_\eta$, \cref{lem:basic}\,\ref{lem:basic:GI} gives
    \begin{align*}
        f(x_{k+1})
        \leq
        f(x_k)-\left(1-\frac{\eta}{2}\right)
        \min\left\{
            1,
            \frac{\grad f(x_k)^\top d_k}{L_k\|d_k\|_2^2}
        \right\}\grad f(x_k)^\top d_k.
    \end{align*}
    Since $h_{t+1}>h_t$, the objective values are nonincreasing along the sequence, and therefore
    \begin{align*}
        \Delta_{t+1}
        =
        f(x_{h_{t+1}})-f(x^\star)
        \leq
        f(x_{k+1})-f(x^\star).
    \end{align*}
    Using \cref{cond:sub}\,\ref{cond:convex}, \cref{lemma:Lipschitz} and \cref{cond:sub}\,\ref{cond:d_t}, we obtain
    \begin{align*}
        \grad f(x_k)^\top d_k
        \geq
        \frac{\Delta_t}{R},
        \qquad
        \frac{\grad f(x_k)^\top d_k}{L_k\|d_k\|_2^2}
        \geq
        \frac{\Delta_t}{RLD_{\cA}^2}.
    \end{align*}
    Hence, we may conclude that
    \begin{align*}
        \Delta_{t+1}
        \leq
        \Delta_t-\frac{2-\eta}{2R}
        \min\left\{
            1,
            \frac{\Delta_t}{RLD_{\cA}^2}
        \right\}\Delta_t
        =
        \max\left\{
            1 - \alpha,
            1 - \beta \Delta_t
        \right\}\Delta_t.
    \end{align*}

    We now prove by induction that $\Delta_t\leq \frac{C}{t+\delta}$.
    For $t=0$, since $h_0\geq 0$ and the objective values are nonincreasing, $\Delta_0 = f(x_{h_0})-f(x^\star) \leq f(x_0)-f(x^\star) \leq \frac{C}{\delta}$.
    Thus the claim holds at $t=0$.

    Assume now that $\Delta_t\leq \frac{C}{t+\delta}$ for some $t\geq 0$. If $\Delta_t\leq \frac{C}{t+\delta+1}$, then $\Delta_{t+1}\leq \Delta_t\leq \frac{C}{t+\delta+1}$, and there is nothing to prove.
    Otherwise, $\Delta_t>\frac{C}{t+\delta+1}$. Since $\beta C \geq 1$, we have
    \begin{align*}
        1-\beta\Delta_t
        <
        1-\frac{1}{t+\delta+1}
        =
        \frac{t+\delta}{t+\delta+1},
        \qquad 
        1-\alpha
        =
        \frac{\delta}{\delta+1}
        \leq
        \frac{t+\delta}{t+\delta+1}.
    \end{align*}
    Therefore, we may conclude that
    \begin{align*}
        \Delta_{t+1}
        \leq
        \max\left\{
            1-\alpha,
            1-\beta\Delta_t
        \right\}\Delta_t 
        \leq
        \frac{t+\delta}{t+\delta+1}\cdot \frac{C}{t+\delta}
        =
        \frac{C}{t+\delta+1}.
    \end{align*}
    This closes the induction.
    Recalling the definitions of $\Delta_t$, $C$ and $\delta$, and using the induction claim conclude the proof.
\end{proof}

For the subroutines studied in \cref{alg:fw-sub,alg:mp-sub,alg:pairwise-sub,alg:away-sub}, the lower bounds \cref{lem:fw-sub,lem:away-sub,lem:mp-sub,lem:pairwise-sub} imply that $|\cG \cap \cI_\eta \cap [t]|$ grows linearly in $t$. By \cref{lem:conv:translate}, any rate of convergence established along the subsequence $\{x_{h_t}\}_{t \geq 0}$ is therefore inherited by the full sequence $\{x_t\}_{t \geq 0}$, up to a shift in the iteration index. We record this in the following corollary, stated for a generic rate $r(\cdot)$ so that it applies uniformly to \cref{thm:convex} and the accelerated rates in the subsequent section.
The proof follows directly from \cref{lem:conv:translate} and is omitted for brevity.

\begin{corollary}
\label{cor:convex-full-sequence}
    Under assumptions of \cref{thm:convex}, suppose that there exist constants $a > 0$ and $b \geq 0$ such that $|\cG \cap \cI_\eta \cap [t]| \geq at - b$ for all $t \ge 0$ and a nonincreasing function $r : \mathbb{N} \to \mathbb{R}_+$ such that
    \begin{align*}
        f(x_{h_t}) - f(x^\star) \leq r(t)
        \qquad \text{for all } t \geq 0.
    \end{align*}
    Then, for all $t \geq (b+1)/a$,
    \begin{align*}
        f(x_t) - f(x^\star) \leq r\bigl(\lceil at - b \rceil - 1\bigr).
    \end{align*}
\end{corollary}

\section{Acceleration under Specific Subroutines}
\label{sec:acceleration}

The results of Section~\ref{sec:convergence} apply uniformly to all admissible subroutines, but they do not fully exploit the geometry of specific variants. 
In this section we show that, under stronger assumptions, particular subroutines admit accelerated rates. As in \cref{sec:convex-convergence}, for fixed $\eta > 1$ that will be clear from context, we will denote by $\{h_t\}_{t \geq 0}$ the indices in $\cG \cap \cI_\eta$ in increasing order.

\subsection{Acceleration under Closed-Loop Subroutine}
\label{sec:acc:FW}
The acceleration results in this subsection exploit geometric information that is not used in the abstract analysis of Sections~\ref{sec:framework} and~\ref{sec:convergence}. Throughout this subsection, the feasible set is $\cX=\conv(\cA)$, and we impose the following additional assumptions.

\begin{assumption}[Strongly convex domain]\label{ass:sc-X}
    The feasible set $\cX$ is $\alpha$-strongly convex for some $\alpha > 0$, that is, for any $\rho \in [0,1]$, $x,y \in \cX$ and $z \in \R^n$ with $\|z\|_2 \leq 1$, we have
    \begin{align*}
        \rho x+(1-\rho)y +\frac{\rho(1-\rho)\alpha \|x-y\|_2^2}{2} z \in \cX.
    \end{align*}
\end{assumption}

\begin{assumption}[Strongly convex objective]\label{ass:sc-f}
    The objective function $f$ is $\mu$-strongly convex for some $\mu > 0$, that is, for any $x,y \in \R^n$, we have
    \begin{align*}
        f(y) \geq f(x)+\grad f(x)^\top (y-x)+ \frac{\mu}{2}\|y-x\|_2^2.
    \end{align*}
\end{assumption}

\begin{assumption}[Nonzero gradient]\label{ass:grad}
    There exists a constant $g > 0$ such that $\|\grad f(x)\|_2 \geq g$ for any $x \in \cX$.
\end{assumption}

The next theorem shows that, under different additional assumptions, the sequence $\{x_{h_t}\}_{t\geq 0}$ enjoys either the accelerated sublinear rate $\cO(1/t^2)$ or a linear convergence rate $\cO(\rho^t)$ for some $\rho \in (0,1)$. 

\begin{theorem}
\label{thm:accelerate}
    Suppose $\{x_t\}_{t \geq 0}$ is the sequence generated by \cref{alg:ac-fw} under the closed-loop Frank-Wolfe subroutine in \cref{alg:fw-sub}, let $\eta \!\in\! (1,2)$. Suppose further that \cref{ass:X-compact,ass:convex-functions,ass:smooth-functions,ass:sc-X,cond:r_t} hold.  
    \begin{enumerate}[label=(\roman*)]
        \item \label{thm:accelerate:sublinear} If additionally \cref{ass:sc-f} holds, then for any $t \geq 0$,
        \begin{align*}
            f(x_{h_t}) - f(x^\star) 
            &\leq
            \frac{
                \max\left\{
                    \frac{LD_{\cA}^2}{2},
                    18\left(\frac{1}{1+\sqrt{\frac{\eta}{2}}}\right)^2
                    \left(\frac{\alpha\sqrt{\mu}}{8\sqrt{2}L}\right)^{-2}
                \right\}
                \left(\frac{1}{1-\sqrt{\frac{\eta}{2}}}\right)^2
            }{
                \left(
                    t + \frac{1}{1-\sqrt{\frac{\eta}{2}}}
                \right)^2
            }.
        \end{align*}
        \item \label{thm:accelerate:linear} If additionally \cref{ass:grad} holds, then for any $t \geq 0$, 
        \begin{align*}
            f(x_{h_t}) - f(x^\star) 
            \leq
            \frac{LD_{\cA}^2}{2}
            \left(
                \max\left\{
                    \frac{\eta}{2},
                    1-\left(1-\frac{\eta}{2}\right)\frac{\alpha g}{8L}
                \right\}
            \right)^{t}.
        \end{align*}
    \end{enumerate}
\end{theorem}

Since the closed-loop Frank-Wolfe subroutine satisfies $\gamma_t^{\max}=1$ at every iteration, we have $\cG = \mathbb N \cup \{0\}$, and therefore $\cG \cap \cI_\eta = \cI_\eta$. Consequently, $\{h_t\}_{t \geq 0}$ simply denotes the indices in $\cI_\eta$.  
The proof of \cref{thm:accelerate} is based on the following one-step recursion, which strengthens the generic descent estimate from \cref{lem:basic} by exploiting the strong convexity of the feasible~set.

\begin{lemma} \label{lemma:intermediate-identity}
    Under the common assumptions of \cref{thm:accelerate}, for any $t \geq 0$,
    \begin{align*}
        f(x_{h_{t+1}}) - f(x^\star)
        \leq
        \max\left\{
            \frac{\eta}{2},
            1 - \left(1-\frac{\eta}{2}\right)\frac{\alpha\|\grad f(x_{h_t})\|_2}{8L}
        \right\}
        (f(x_{h_t})-f(x^\star)).
    \end{align*}
    If additionally \cref{ass:sc-f} holds, then for every $t\geq 0$,
    \begin{align*}
        f(x_{h_{t+1}})-f(x^\star)
        \leq
        \max\left\{
            \frac{\eta}{2},
            1-\left(1-\frac{\eta}{2}\right)\frac{\alpha\sqrt{\mu\bigl(f(x_{h_t})-f(x^\star)\bigr)}}{8\sqrt{2}L}
        \right\}
        \bigl(f(x_{h_t})-f(x^\star)\bigr).
    \end{align*}
\end{lemma}

\begin{proof}[Proof of \cref{lemma:intermediate-identity}]
    Fix $t\geq 0$. By \cref{ass:sc-X}, for any $z\in\R^n$ with $\|z\|_2\leq 1$,
    \begin{align*}
        \frac{1}{2}v_{h_t} + \frac{1}{2}x_{h_t} + \frac{\alpha}{8} \|v_{h_t} - x_{h_t}\|_2^2 z
        \in
        \conv(\cA) = \cX.
    \end{align*}
    Since $v_{h_t}$ minimizes $u \mapsto \grad f(x_{h_t})^\top u$ over $\cA$, it also minimizes it over $\conv(\cA)$. Therefore,
    \begin{align*}
        \grad f(x_{h_t})^\top v_{h_t}
        \leq
        \grad f(x_{h_t})^\top\left(
            \frac{1}{2}v_{h_t} + \frac{1}{2}x_{h_t} + \frac{\alpha}{8} \|v_{h_t} - x_{h_t}\|_2^2 z
        \right).
    \end{align*}
    Rearranging gives
    \begin{align*}
        \grad f(x_{h_t})^\top (v_{h_t} - x_{h_t})
        \leq
        \frac{1}{2}\grad f(x_{h_t})^\top (v_{h_t} - x_{h_t})
        +\left( \frac{\alpha}{8}\|v_{h_t} - x_{h_t}\|_2^2 \right) z^\top \grad f(x_{h_t}).
    \end{align*}
    Since $\cX=\conv(\cA)$, \cref{lem:basic}\,\ref{lem:basic:fw-gap} yields $\grad f(x_{h_t})^\top (v_{h_t} - x_{h_t}) \leq f(x^\star)-f(x_{h_t}) \leq 0$.
    Hence,
    \begin{align*}
        \grad f(x_{h_t})^\top (v_{h_t} - x_{h_t})
        \leq
        \left(\frac{\alpha}{8}\|v_{h_t} - x_{h_t}\|_2^2 \right) z^\top \grad f(x_{h_t}),
        \qquad \forall z \in \R^n,\ \|z\|_2 \leq 1.
    \end{align*}
    Minimizing the right-hand side over all $\|z\|_2\leq 1$ and using $d_{h_t} = x_{h_t} - v_{h_t}$, we obtain
    \begin{align*}
        \grad f(x_{h_t})^\top d_{h_t}
        \geq
        \frac{\alpha}{8} \|d_{h_t}\|_2^2 \|\grad f(x_{h_t})\|_2
        \quad \implies \quad
        \frac{\grad f(x_{h_t})^\top d_{h_t}}{L_{h_t} \|d_{h_t}\|_2^2}
        \geq
        \frac{\alpha\|\grad f(x_{h_t})\|_2}{8 L_{h_t}}
        \geq
        \frac{\alpha\|\grad f(x_{h_t})\|_2}{8L},
    \end{align*}
    where the last inequality follows from \cref{lemma:Lipschitz}.
    Since $h_t \in \cG \cap \cI_\eta$, \cref{lem:basic}\,\ref{lem:basic:GI} gives
    \begin{align*}
        f(x_{h_t+1})
        \leq
        f(x_{h_t}) - \left(1-\frac{\eta}{2}\right)
        \min\left\{
            1,
            \frac{\grad f(x_{h_t})^\top d_{h_t}}{L_{h_t}\|d_{h_t}\|_2^2}
        \right\}
        \grad f(x_{h_t})^\top d_{h_t}.
    \end{align*}
    Using the lower bound above and \cref{lem:basic}\,\ref{lem:basic:fw-gap}, we obtain
    \begin{align*}
        f(x_{h_t+1})-f(x^\star)
        &\leq
        f(x_{h_t})-f(x^\star)
        -\left(1-\frac{\eta}{2}\right)
        \min\left\{
            1,
            \frac{\alpha\|\grad f(x_{h_t})\|_2}{8L}
        \right\}
        \bigl(f(x_{h_t})-f(x^\star)\bigr).
    \end{align*}
    Since the objective values are nonincreasing and $h_{t+1}>h_t$, we have $f(x_{h_{t+1}}) \leq f(x_{h_t+1})$, and therefore
    \begin{align*}
        f(x_{h_{t+1}})-f(x^\star)
        \leq
        \max\left\{
            \frac{\eta}{2},
            1-\left(1-\frac{\eta}{2}\right)\frac{\alpha\|\grad f(x_{h_t})\|_2}{8L}
        \right\}
        \bigl(f(x_{h_t})-f(x^\star)\bigr),
    \end{align*}
    which proves the first claim.

    Suppose now that \cref{ass:sc-f} also holds. Since $x^\star$ minimizes $f$ over $\conv(\cA)$ and $f$ is differentiable, the first-order optimality condition gives $\grad f(x^\star)^\top (y-x^\star)\geq 0$ for all $y \in \conv(\cA)$.
    Using strong convexity of $f$, we therefore obtain $f(y)-f(x^\star)\geq \frac{\mu}{2}\|y-x^\star\|_2^2$ for all $y \in \conv(\cA)$.
    On the other hand, by convexity of $f$, $f(y)-f(x^\star)\leq \grad f(y)^\top (y-x^\star)\leq \|\grad f(y)\|_2\|y-x^\star\|_2.$
    Combining the two estimates yields $\|\grad f(y)\|_2\geq \sqrt{\frac{\mu}{2}\bigl(f(y)-f(x^\star)\bigr)}$ for all $y \in \conv(\cA)$.
    Applying this with $y=x_{h_t}$ in the first claim proves the second statement.
\end{proof}

Armed with \cref{lemma:intermediate-identity}, we are now ready to prove \cref{thm:accelerate}.

\begin{proof}[Proof of \cref{thm:accelerate}]
    We first prove \ref{thm:accelerate:sublinear}. Define
    \begin{align*}
        M:=\left(1-\frac{\eta}{2}\right)\frac{\alpha\sqrt{\mu}}{8\sqrt{2}L},
        \qquad
        \delta:=\frac{1}{1-\sqrt{\frac{\eta}{2}}}-3,
        \qquad
        C:= \max\left\{
            \left(\frac{1}{1-\sqrt{\frac{\eta}{2}}}\right)^2\frac{LD_{\cA}^2}{2},
            18M^{-2}
        \right\}.
    \end{align*}
    We claim that, for every $t\geq 0$,
    \begin{align}
    \label{eq:induction:sc}
        f(x_{h_t})-f(x^\star)\leq \frac{C}{(t+\delta+3)^2}.
    \end{align}
    For $t=0$, by smoothness, we have
    \begin{align*}
        f(x_0)
        \leq
        f(x_{-1})+\grad f(x_{-1})^\top(x_0-x_{-1})+\frac{L}{2}\|x_0-x_{-1}\|_2^2.
    \end{align*}
    Since $x_0$ minimizes $v\mapsto \grad f(x_{-1})^\top v$ over $\cA$ and $x^\star \in \conv(\cA)$, we have $\grad f(x_{-1})^\top x_0\leq \grad f(x_{-1})^\top x^\star$.
    Therefore, we may conclude that
    \begin{align*}
        f(x_0) - f(x^\star)
        \leq
        f(x_{-1})-f(x^\star)+\grad f(x_{-1})^\top(x^\star-x_{-1})+\frac{L}{2}\|x_0-x_{-1}\|_2^2 
        \leq
        \frac{LD_{\cA}^2}{2},
    \end{align*}
    where the last step uses convexity of $f$ and $x_{-1},x_0\in\cA$. Since $h_0\geq 0$ and the objective values are nonincreasing along the sequence. This implies 
    \begin{align}
    \label{eq:initial}
        f(x_{h_0})-f(x^\star) \leq f(x_0) - f(x^\star) \leq \frac{LD_{\cA}^2}{2} \leq \frac{C}{(\delta+3)^2},
    \end{align}
    as required.
    Assume now that~\eqref{eq:induction:sc} holds for some $t\geq 0$. If
    \begin{align*}
        f(x_{h_t})-f(x^\star)\leq \frac{C}{2(t+\delta+3)^2}
        \quad \implies \quad
        f(x_{h_{t+1}})-f(x^\star)
        \leq
        f(x_{h_t})-f(x^\star)
        \leq
        \frac{C}{(t+\delta+4)^2},
    \end{align*}
    because $2(t+\delta+3)^2 -(t+\delta+4)^2 =(t+\delta)^2 + 4(t+\delta) + 2 > 0$.

    Assume next that
    \begin{align*}
        f(x_{h_t})-f(x^\star)>\frac{C}{2(t+\delta+3)^2}.
    \end{align*}
    By \cref{lemma:intermediate-identity},
    \begin{align*}
        f(x_{h_{t+1}})-f(x^\star)
        \leq
        \max\left\{
            \frac{\eta}{2},
            1-M\sqrt{f(x_{h_t})-f(x^\star)}
        \right\}
        \bigl(f(x_{h_t})-f(x^\star)\bigr).
    \end{align*}
    If $\frac{\eta}{2}\geq 1-M\sqrt{f(x_{h_t})-f(x^\star)}$, then
    \begin{align*}
        f(x_{h_{t+1}})-f(x^\star)
        \leq
        \frac{\eta}{2}\bigl(f(x_{h_t})-f(x^\star)\bigr) 
        \leq
        \frac{\eta C}{2(t+\delta+3)^2} 
        \leq
        \frac{C}{(t+\delta+4)^2},
    \end{align*}
    where the last inequality follows from $\frac{1}{t+\delta+4}\leq \frac{1}{\delta+3}=1-\sqrt{\frac{\eta}{2}}$.
    
    If instead $\frac{\eta}{2}<1-M\sqrt{f(x_{h_t})-f(x^\star)}$, then
    \begin{align*}
        f(x_{h_{t+1}})-f(x^\star)
        &\leq
        \left(1-M\sqrt{f(x_{h_t})-f(x^\star)}\right)\bigl(f(x_{h_t})-f(x^\star)\bigr) 
        <
        \left(1-M\sqrt{\frac{C}{2}}\frac{1}{t+\delta+3}\right)\frac{C}{(t+\delta+3)^2}.
    \end{align*}
    To show that this is at most $\frac{C}{(t+\delta+4)^2}$, it suffices to prove
    \begin{align*}
        \frac{(t+\delta+4)^2}{(t+\delta+3)^2}
        \left(1-M\sqrt{\frac{C}{2}}\frac{1}{t+\delta+3}\right)\leq 1.
    \end{align*}
    This is equivalent to proving
    \begin{align*}
        \frac{(2(t+\delta)+7)(t+\delta+3)}{(t+\delta+4)^2}
        \leq
        M\sqrt{\frac{C}{2}}.
    \end{align*}
    By the definition of $C$, we have $M \sqrt{{C}/{2}} \geq M \sqrt{9M^{-2}}=3$.
    On the other hand,
    \begin{align*}
        \frac{(2(t+\delta)+7)(t+\delta+3)}{(t+\delta+4)^2}-3
        =
        \frac{-(t+\delta)^2-11(t+\delta)-27}{(t+\delta+4)^2}<0.
    \end{align*}
    Hence, the required inequality holds, and the induction is complete. This proves \ref{thm:accelerate:sublinear}.

    
    We next prove \ref{thm:accelerate:linear}. Define
    \begin{align*}
        \rho:=
        \max\left\{
            \frac{\eta}{2},
            1-\left(1-\frac{\eta}{2}\right)\frac{\alpha g}{8L}
        \right\}.
    \end{align*}
    Since $\eta\in(1,2)$ and $g>0$, we have $\rho<1$. We claim that, for every $t\geq 0$,
    \begin{align}
    \label{eq:induction:linear}
        f(x_{h_t})-f(x^\star)\leq \frac{LD_{\cA}^2}{2}\rho^t.
    \end{align}
    We prove this by induction. 
    For $t=0$, the initial bound~\eqref{eq:initial} gives $f(x_{h_0})-f(x^\star) \leq {LD_{\cA}^2}/{2}.$
    Thus, the base case of the induction holds.

    Assume~\eqref{eq:induction:linear} holds for some $t \geq 0$. Then, by \cref{lemma:intermediate-identity} and \cref{ass:grad},
    \begin{align*}
        f(x_{h_{t+1}})-f(x^\star)
        &\leq
        \max\left\{
            \frac{\eta}{2},
            1-\left(1-\frac{\eta}{2}\right)\frac{\alpha\|\grad f(x_{h_t})\|_2}{8L}
        \right\}
        \bigl(f(x_{h_t})-f(x^\star)\bigr) \\
        &\leq
        \rho \left(f(x_{h_t})-f(x^\star)\right) \\
        &\leq
        \frac{LD_{\cA}^2}{2}\rho^{t+1}.
    \end{align*}
    This concludes the induction and completes the proof of \ref{thm:accelerate:linear}.
\end{proof}

For the closed-loop Frank-Wolfe subroutine in \cref{alg:fw-sub}, the lower bound established in \cref{lem:fw-sub} shows that $|\cG \cap \cI_\eta \cap[t]|$ grows linearly with $t$. 
Therefore, by \cref{cor:convex-full-sequence}, convergence estimates obtained along the subsequence $\{x_{h_t}\}_{t\geq 0}$ transfer directly to the full sequence $\{x_t\}_{t\geq 0}$. 
In particular, for all sufficiently large $t$, the full sequence inherits the same $\cO(1/t^2)$ or $\cO(\rho^t)$ rate.

\subsection{Acceleration under Matching Pursuit, Pairwise, and Away-Step Subroutines}
\label{sec:acc:MP-PFW-AFW}

We now turn to linear convergence under stronger assumptions for the Matching Pursuit, pairwise Frank-Wolfe, and away-step Frank-Wolfe subroutines. 
In all three cases, the key additional ingredient is a quadratic error bound relating the objective gap to the square of the directional gap generated by the corresponding subroutine. 
We first introduce the geometric quantities that appear in these bounds, and then state the main theorem of this subsection.

\begin{definition}[Geometric strong convexity]
\label{def:geo-strong-convexity}
The geometric strong convexity constant $\mu_\cA$ of $f$ over $\conv(\cA)$ is defined as
\begin{align*}
    \mu_\cA
    :=
    \inf\left\{
        \frac{2\bigl(f(y)-f(x)-\grad f(x)^\top (y-x)\bigr)}{\gamma_{\cA}(y,x)^2}
        :
        x,y\in\conv(\cA),
        \ \grad f(x)^\top (y-x)<0
    \right\},
\end{align*}
where
\begin{align*}
    &\gamma_{\cA}(y,x)
    :=
    \frac{-\grad f(x)^\top (y-x)}{\grad f(x)^\top (s_{\cA}(x)-v_{\cA}(x))}, \qquad \text{with} \\[1ex]
    &v_{\cA}(x)\in\argmin_{v\in\cA}\grad f(x)^\top v,\\
    &s_{\cA}(x)\in\argmin\left\{
        \grad f(x)^\top s
        :
        s\in \argmax_{u\in \cS}\grad f(x)^\top u,\ \cS\in \cS_x
    \right\}, \\
    &\cS_x
    :=
    \left\{
        \cS\subseteq \cA : x \text{ is a proper convex combination of all atoms in } \cS
    \right\}.
\end{align*}
\end{definition}

\begin{definition}[Minimal directional width]
\label{def:mdw}
The minimal directional width of $\cA$ is
\begin{align*}
    w_\cA
    :=
    \min_{d\in\lin(\cA)\setminus\{0\}}
    \max_{z\in\cA}
    \frac{z^\top d}{\|d\|_2}.
\end{align*}
\end{definition}

Under \cref{ass:sc-f,ass:atom}, the constant $\mu_\cA$ is strictly positive by \citep[Theorem 6]{lacoste2015global}. Let $\{h_t\}_{t\geq 0}$ denote the elements of $\cG\cap\cI_\eta$ in increasing order.

\begin{theorem}
\label{thm:linear-subroutines}
    Suppose $\{x_t\}_{t \geq 0}$ is the sequence generated by \cref{alg:ac-fw} under one of the subroutines \cref{alg:mp-sub,alg:pairwise-sub,alg:away-sub}, and let $\eta\in(1,2)$. Assume \cref{ass:X-compact,ass:smooth-functions,ass:sc-f} together with Condition~\ref{cond:r_t}. In addition, suppose \cref{ass:mp} holds under the Matching Pursuit subroutine, and \cref{ass:atom} holds under the pairwise and away-step subroutines. Then, for any $t\geq 0$,
    \begin{align*}
        f(x_{h_t})-f(x^\star)
        \leq
        \bigl(f(x_0)-f(x^\star)\bigr)
        \max\left\{
            1-\left(1-\frac{\eta}{2}\right)\frac{1}{R},
            1-\left(1-\frac{\eta}{2}\right)\frac{1}{G L D_{\cA}^2}
        \right\}^t,
    \end{align*}
    where $G = 1/(2\mu w_\cA^2)$ for the Matching Pursuit subroutine, $G = 1/(2\mu_\cA)$ for the pairwise Frank-Wolfe subroutine, and $G = 2/\mu_\cA$ for the away-step Frank-Wolfe subroutine.
\end{theorem}

The proof of \cref{thm:linear-subroutines} relies on the following subroutine-specific quadratic error bounds. These estimates follow from \citep{pedregosa2020linearly} and \citep{lacoste2015global}; to keep the presentation self-contained, we include the reductions to our notation.

\begin{lemma}
\label{lem:quadratic-gap}
Suppose \cref{ass:smooth-functions,ass:sc-f} hold.
\begin{enumerate}[label=(\roman*)]
    \item \label{lem:quadratic-gap:mp}
    If additionally \cref{ass:mp} holds, $x\in\lin(\cA)$, $v\in\argmin_{u\in\cA}\grad f(x)^\top u$, and $d=-v$, then
    \begin{align*}
        f(x)-f(x^\star)
        \leq
        \frac{\bigl(\grad f(x)^\top d\bigr)^2}{2 \mu w_\cA^2}.
    \end{align*}

    \item \label{lem:quadratic-gap:pairwise}
    If additionally \cref{ass:atom} holds, $x\in\conv(\cA)$, $\cS\in\cS_x$, $v\in\argmin_{u\in\cA}\grad f(x)^\top u$, $s\in\argmax_{u\in\cS}\grad f(x)^\top u$, and $d=s-v$, then
    \begin{align*}
        f(x)-f(x^\star)
        \leq
        \frac{\bigl(\grad f(x)^\top d\bigr)^2}{2\mu_\cA}.
    \end{align*}

    \item \label{lem:quadratic-gap:away}
    If additionally \cref{ass:atom} holds, $x\in\conv(\cA)$, $\cS \in \cS_x$, $v\in\argmin_{u\in\cA}\grad f(x)^\top u$, $s\in\argmax_{u\in\cS}\grad f(x)^\top u$, and $d$ is chosen by the away-step rule, then
    \begin{align*}
        f(x)-f(x^\star)
        \leq
        \frac{2\bigl(\grad f(x)^\top d\bigr)^2}{\mu_\cA}.
    \end{align*}
\end{enumerate}
\end{lemma}

\begin{proof}[Proof of \cref{lem:quadratic-gap}]
    For \ref{lem:quadratic-gap:mp}, let $\grad_{\cA}f(x)$ denote the orthogonal projection of $\grad f(x)$ onto $\lin(\cA)$, and let $\|\cdot\|_{\cA^*}$ be the dual norm of the atomic norm $\|\cdot\|_{\cA}$. By \citep[Lemma 9]{pedregosa2020linearly}, we have $f(x)-f(x^\star) \leq \frac{\|\grad_{\cA}f(x)\|_{\cA^*}^2}{2 \mu w_\cA^2}.$
    Since $x\in\lin(\cA)$ and $v\in\cA\subseteq\lin(\cA)$, we have
    \begin{align*}
        \grad f(x)^\top v=\grad_{\cA}f(x)^\top v.
    \end{align*}
    Moreover, under \cref{ass:mp} the dictionary is symmetric, so
    \begin{align*}
        \|\grad_{\cA}f(x)\|_{\cA^*}
        =
        \max_{\|z\|_{\cA}\leq 1}\grad_{\cA}f(x)^\top z
        =
        -\min_{v\in\cA}\grad_{\cA}f(x)^\top v
        =
        -\min_{v\in\cA}\grad f(x)^\top v
        =
        \grad f(x)^\top d.
    \end{align*}
    Substituting this identity into the previous inequality proves \ref{lem:quadratic-gap:mp}.
    For \ref{lem:quadratic-gap:pairwise}, \citep[Theorem 4.A]{pedregosa2020linearly} gives
    \begin{align*}
        f(x)-f(x^\star)
        \leq
        \frac{\bigl(\grad f(x)^\top (s_{\cA}(x)-v_{\cA}(x))\bigr)^2}{2\mu_\cA}.
    \end{align*}
    Since $\cS\in\cS_x$, the current active set is one admissible choice in the definition of $s_{\cA}(x)$. Hence, $\grad f(x)^\top s_{\cA}(x)\leq \grad f(x)^\top s$.
    Also, by definition, $v=v_{\cA}(x)$. Therefore,
    \begin{align*}
        \grad f(x)^\top (s_{\cA}(x)-v_{\cA}(x))
        \leq
        \grad f(x)^\top (s-v)
        =
        \grad f(x)^\top d.
    \end{align*}
    This proves \ref{lem:quadratic-gap:pairwise}.
    For \ref{lem:quadratic-gap:away}, \citep[Theorem 4.A]{pedregosa2020linearly} also yields
    \begin{align*}
        f(x)-f(x^\star)
        \leq
        \frac{2\left(\max\left\{
            \grad f(x)^\top (s_{\cA}(x)-x),
            \grad f(x)^\top (x-v_{\cA}(x))
        \right\}\right)^2}{\mu_\cA}.
    \end{align*}
    Since $\cS\in\cS_x$, we again have $\grad f(x)^\top (s_{\cA}(x)-x) \leq \grad f(x)^\top (s-x)$, while $v=v_{\cA}(x)$. Hence,
    \begin{align*}
        \max\left\{
            \grad f(x)^\top (s_{\cA}(x)-x),
            \grad f(x)^\top (x-v_{\cA}(x))
        \right\}
        \leq
        \max\left\{
            \grad f(x)^\top (s-x),
            \grad f(x)^\top (x-v)
        \right\}
        =
        \grad f(x)^\top d.
    \end{align*}
    Substituting this into the previous inequality proves \ref{lem:quadratic-gap:away}.
\end{proof}

Armed with \cref{lem:quadratic-gap}, we are ready to prove \cref{thm:linear-subroutines}.

\begin{proof}[Proof of \cref{thm:linear-subroutines}]
    Define
    \begin{align*}
        \Delta_t:=f(x_{h_t})-f(x^\star),
        \qquad
        g_t:=\grad f(x_{h_t})^\top d_{h_t}.
    \end{align*}
    Since $h_t\in\cG\cap\cI_\eta$, \cref{lem:basic}\,\ref{lem:basic:GI} gives
    \begin{align*}
        f(x_{h_t+1})
        \leq
        f(x_{h_t})
        -\left(1-\frac{\eta}{2}\right)
        \min\left\{
            1,
            \frac{g_t}{L_{h_t}\|d_{h_t}\|_2^2}
        \right\}g_t.
    \end{align*}
    Since the objective values are nonincreasing and $h_{t+1}>h_t$, we have
    \begin{align*}
        \Delta_{t+1}
        =
        f(x_{h_{t+1}})-f(x^\star)
        \leq
        f(x_{h_t+1})-f(x^\star)
        \quad \implies \quad
        \Delta_{t+1}
        \leq
        \Delta_t
        -\left(1-\frac{\eta}{2}\right)
        \min\left\{
            1,
            \frac{g_t}{L_{h_t}\|d_{h_t}\|_2^2}
        \right\}g_t.
    \end{align*}
    By \cref{lemma:Lipschitz} and \cref{cond:sub}\,\ref{cond:d_t}, $L_{h_t} \leq L$ and $\|d_{h_t}\|_2 \leq D_{\cA}$. Thus, we have
    \begin{align*}
        \Delta_{t+1}
        \leq
        \Delta_t
        -\left(1-\frac{\eta}{2}\right)
        \min\left\{
            1,
            \frac{g_t}{L D_{\cA}^2}
        \right\}g_t.
    \end{align*}
    By \cref{cond:sub}\,\ref{cond:convex}, $g_t \geq {\Delta_t}/{R}$.
    On the other hand, by \cref{lem:quadratic-gap}, we have $\Delta_t \leq G g_t^2$.
    Hence, $g_t^2 \geq \frac{\Delta_t}{G}$ and we obtain
    \begin{align*}
        \Delta_{t+1}
        &\leq
        \max\left\{
            1-\left(1-\frac{\eta}{2}\right)\frac{1}{R},
            1-\left(1-\frac{\eta}{2}\right)\frac{1}{G L D_{\cA}^2}
        \right\}\Delta_t.
    \end{align*}
    Iterating this recursion yields
    \begin{align*}
        \Delta_t
        \leq
        \Delta_0
        \max\left\{
            1-\left(1-\frac{\eta}{2}\right)\frac{1}{R},
            1-\left(1-\frac{\eta}{2}\right)\frac{1}{G L D_{\cA}^2}
        \right\}^t.
    \end{align*}
    Since $h_0\geq 0$ and the objective values are nonincreasing, $\Delta_0 = f(x_{h_0})-f(x^\star) \leq f(x_0)-f(x^\star)$.
    Substituting this bound into the previous inequality concludes the proof.
\end{proof}

\section{Numerical Results}
\label{sec:numerical}
We compare variants of Frank-Wolfe methods across five problem classes, organized into four algorithmic families: the closed-loop Frank-Wolfe method (\texttt{CFW}), the away-step Frank-Wolfe method (\texttt{AFW}), the pairwise Frank-Wolfe method (\texttt{PFW}), and the Matching Pursuit method (\texttt{MP}). For each family, we evaluate three stepsize rules: the classic variant that uses the global Lipschitz constant of the gradient (no prefix), the backtracking line-search in \citep{pedregosa2020linearly} (prefix \texttt{B-}), and our auto-conditioned scheme proposed in \cref{alg:ac-fw} (prefix \texttt{AC-}). For reference, we also include the open-loop Frank-Wolfe method (\texttt{OFW}), corresponding to the Frank-Wolfe subroutine with $\gamma_t = 2/(t+2)$ in convex setting and $\gamma_t = 1 / \sqrt{t+1}$.
In the following, we first describe the problem formulations and datasets and then present the numerical results. To ensure reproducibility, all source codes are made available at \url{https://github.com/brucegiang/AC-FW}.

\paragraph{Dictionary Learning.}
The goal of dictionary learning is to find a compact representation of a dataset $A = \{a_1, \dots, a_n\} \subseteq \R^m$. Formally, we seek a dictionary $D = [d_1 \cdots d_p] \in \R^{m \times p}$ such that each data point $a_i$ is well approximated by a linear combination of the columns of $D$. This leads to the problem
\begin{align*}
    \min_{D, X} \left\{ \frac{1}{2n} \sum_{i \in [n]} \|a_i - D x_i\|_2^2 :
    \begin{array}{l}
        D \in \R^{m \times p}, ~ X = [x_1 \cdots x_n] \in \R^{p \times n}, \\[1ex]
        \|d_j\|_2 \leq 1, ~ \forall j \in [p], \quad \|x_i\|_1 \leq \beta, ~ \forall i \in [n]
    \end{array}
     \right\},
\end{align*}
where $X$ is the coefficient matrix. We generate the data as follows. We set $n = 1000$, $m = 500$, and $p = 200$, and construct the dataset matrix $A = B C \in \R^{m \times n}$, where $B \in \R^{m \times p}$ has entries drawn i.i.d.\ from the standard Gaussian distribution with columns rescaled to unit $\ell_2$-norm, and $C = \tfrac{1}{\|U\|_2 \|V\|_2} U V^\top$ with $U \in \R^{p \times l}$ and $V \in \R^{n \times l}$ also drawn i.i.d.\ from the standard Gaussian distribution and $l = 50$. We test $\beta \in \{5, 10, 15\}$. The data generation follows the procedure in \citep{boroun2023projection,giang2026projection}. We implement all variants of the \texttt{CFW} algorithm, along with \texttt{OFW}, in the nonconvex setting.

\paragraph{Constrained Regularized Logistic Regression Problem.}
In a binary classification setting, we are given a training dataset $\cD^{\text{tr}} = \{(a_i, b_i)\}_{i=1}^n$, where $a_i \in \R^d$ is the feature vector and $b_i \in \{0,1\}$ is the class label. The objective is to learn a predictor $x \in \R^d$ such that, for a new instance $\hat a$, the predicted label is $\hat b = 1$ if $\hat a^\top x \geq 0$ and $\hat b = 0$ otherwise. We consider the following optimization problem
\begin{equation*} \label{eq:l1-logist}
    \min_{\|x\|_1 \leq \beta} ~ \frac{1}{n} \sum_{i \in [n]} \ell(a_i^\top x, b_i) + \frac{\lambda}{2} \|x\|_2^2,
\end{equation*}
where for any $y \in \R$ and $b \in \{0, 1\}$, the logistic loss is given by $\ell(y, b) := \log(1 + e^y) - b y$.
Following \citep{pedregosa2020linearly}, we evaluate our algorithms on the \texttt{RCV1.binary} dataset from the LIBSVM repository \citep{chang2011libsvm} and the \texttt{Madelon} dataset from the UCI repository \citep{madelon_171}, and we set $\lambda = 0.01$ and $\beta \in \{5, 10, 15\}$. We implement all variants of \texttt{CFW}, \texttt{PFW}, and \texttt{AFW}, along with \texttt{OFW}, in the convex setting with a finite dictionary.

\paragraph{Unconstrained Regularized Logistic Regression Problem.}
To evaluate the Matching Pursuit subroutine in \cref{alg:mp-sub}, we also consider the following optimization problem
\begin{equation*} 
    \min_{x \in \lin(\cA)} ~ \frac{1}{n} \sum_{i \in [n]} \ell(a_i^\top x, b_i) + \frac{\lambda}{2} \|x\|_2^2 ,
\end{equation*}
where $\cA = \{x \in \R^d : \|x\|_1 \leq \beta\}$ and $\ell$ denotes the logistic loss. Following \citep{pedregosa2020linearly}, we use the same datasets and parameter settings as in the constrained case. We implement all variants of \texttt{MP} in the convex setting.

\paragraph{Constrained Huber Regression Problem.}
We consider the collaborative filtering problem \citep{mehta2007robust}
\begin{align*}
    \min_{\|X\|_* \leq \beta} ~ \frac{1}{|\cI|} \sum_{(i, j) \in \cI} \ell_\delta(A_{ij} - X_{ij}),
\end{align*}
where $\cI \subseteq [m] \times [p]$ is the set of observed entries of $A$, and $\ell_\delta$ denotes the Huber loss with parameter $\delta > 0$, defined as $\ell_\delta(a) := \tfrac{1}{2} a^2$ if $|a| \leq \delta$ and $\ell_\delta(a) := \delta(|a| - \tfrac{1}{2}\delta)$ otherwise. Following \citep{pedregosa2020linearly}, we use the \texttt{MovieLens} dataset \citep{Harperetal2015} with 1M movie ratings, and set $\delta = 1$ and $\beta \in \{10000, 20000\}$. We implement all variants of \texttt{CFW}, along with \texttt{OFW}, in the convex setting with an infinite dictionary.

\paragraph{Distributionally Robust Minimum Mean Square Error Estimation.}
We consider the distributionally robust minimum mean square error estimation problem studied in \citep{shafieezadeh2018wasserstein,nguyen2023bridging}, which seeks an estimator of a signal $x \in \R^n$ from an observation $y \in \R^m$ under distributional uncertainty modeled by a Wasserstein ambiguity set around a nominal Gaussian distribution $\bbP = \cN_d(0, \Sigma)$ on $\R^d$, where $d = n + m$. As shown in \citep[Theorem 2.5]{shafieezadeh2018wasserstein}, the resulting minimax problem admits the convex reformulation
\begin{align*}
    \max \left\{ 
    \Tr\left[ S_{xx} - S_{xy} S_{yy}^{-1} S_{yx} \right] : S = \begin{bmatrix} S_{xx} & S_{xy} \\ S_{yx} & S_{yy} \end{bmatrix} \in \bbS_+^d, \ \Tr\left[ S + \Sigma - 2 \left( \Sigma^{1/2} S \Sigma^{1/2} \right)^{1/2} \right] \leq \rho^2
    \right\},
\end{align*}
where $\rho > 0$ is the Wasserstein radius.
Let $S^\star$ denote the optimizer of the above problem. Given output $y$, the prediction is $\hat x = S^\star_{yx} (S_{yy}^\star)^{-1} y$. Maybe surprising, the linear minimization oracle over the feasible set admits a quasi-closed solution due to \citep[Theorem~3.2]{shafieezadeh2018wasserstein}.

We generate the data following \citep[Section 5.1]{shafieezadeh2018wasserstein}. We aim to predict a signal $x \in \R^{4d/5}$ from an observation $y \in \R^{d/5}$, with $d \in \{100, 1000\}$. 
We draw a matrix $A \in \R^{d \times d}$ with i.i.d.\ standard Gaussian entries, and let $R$ denotes the orthogonal matrices whose columns are the orthonormal eigenvectors of $A + A^\top$. We then define $\Sigma = R \Lambda R^\top$, where $\Lambda$ is a diagonal matrix with entries drawn uniformly from $[0.1, 10]^d$. The Wasserstein radius is set to $\rho = \sqrt{d}$.  We implement all variants of \texttt{CFW}, along with \texttt{OFW}, in the convex setting with an infinite dictionary.

\paragraph{Summary of Numerical Results.}
Figures~\ref{fig:Experiment1} reports the Frank-Wolfe gap as a function of wall-clock time for all variants across the five problem classes described above. Across nearly all instances, the auto-conditioned variants (\texttt{AC-}) match or outperform their backtracking counterparts (\texttt{B-}), and both consistently improve on the classic variants that rely on a global Lipschitz constant and on the open-loop \texttt{OFW} baseline. The advantage of \texttt{AC-} over \texttt{B-} is most pronounced when objective evaluations are expensive. For example, on the dictionary learning, \texttt{MovieLens}, and distributionally robust MMSE estimation problems, \texttt{AC-CFW} reaches gap levels several orders of magnitude smaller than \texttt{B-CFW} within the same time budget, reflecting the per-iteration savings from avoiding repeated function evaluations during line search. On the logistic regression problems, where function evaluations are comparatively cheap, the gap between \texttt{AC-} and \texttt{B-} narrows, but the auto-conditioned variants remain competitive across all subroutines and choices of $\beta$. The plain variants and \texttt{OFW}, which do not adapt to local curvature, lag substantially in most instances and in some cases stall well above the levels reached by the adaptive methods. Overall, these results indicate that the auto-conditioned scheme delivers the benefits of adaptivity while retaining the simplicity and low per-iteration cost of closed-loop step-size rules. To ensure reproducibility, all
source codes are made available at \url{https://github.com/brucegiang/AC-FW}.

\newpage
\begin{figure*}[!h]
    \centering
    \begin{subfigure}[b]{0.28\textwidth}
        \centering
        \includegraphics[width=\linewidth]{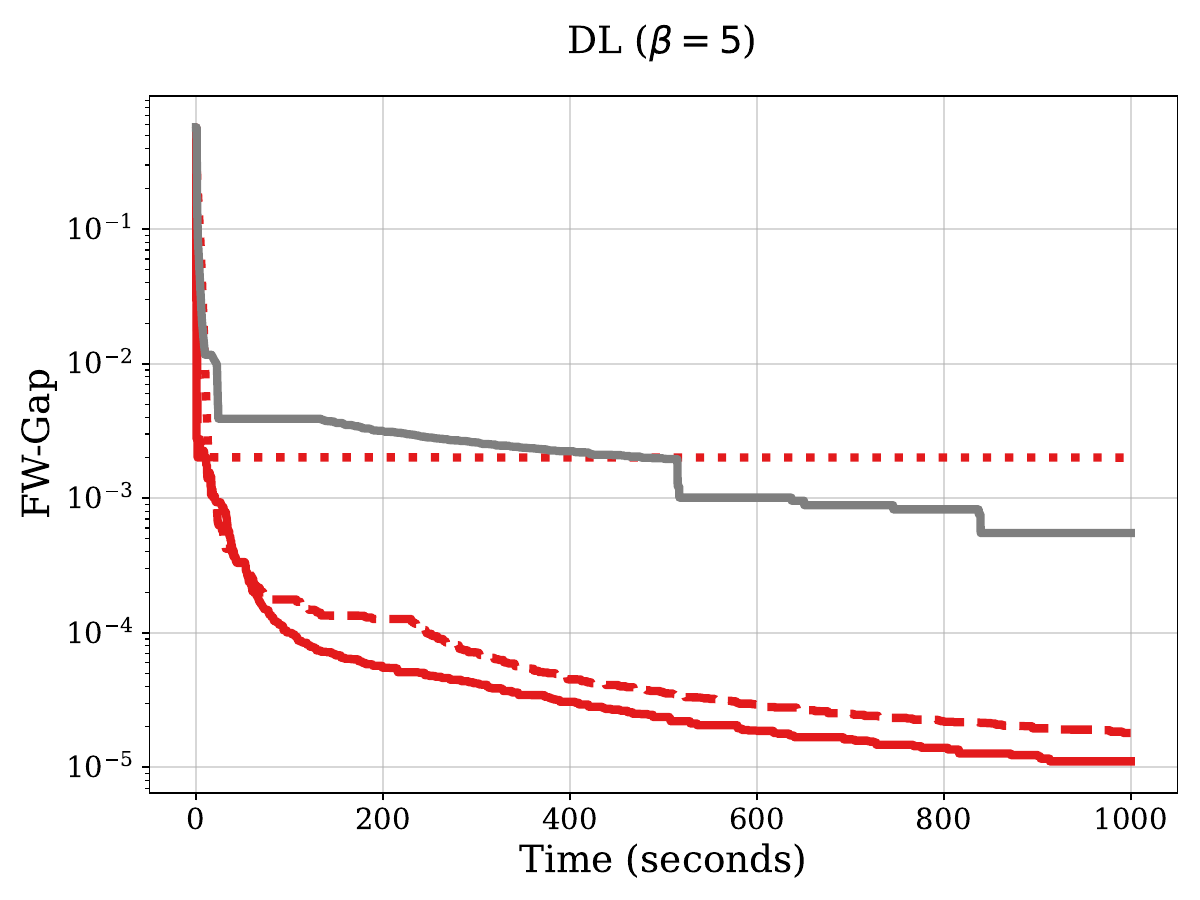}
        \label{subfig:dl5}
    \end{subfigure}
    ~
    \begin{subfigure}[b]{0.28\textwidth}
        \centering
        \includegraphics[width=\linewidth]{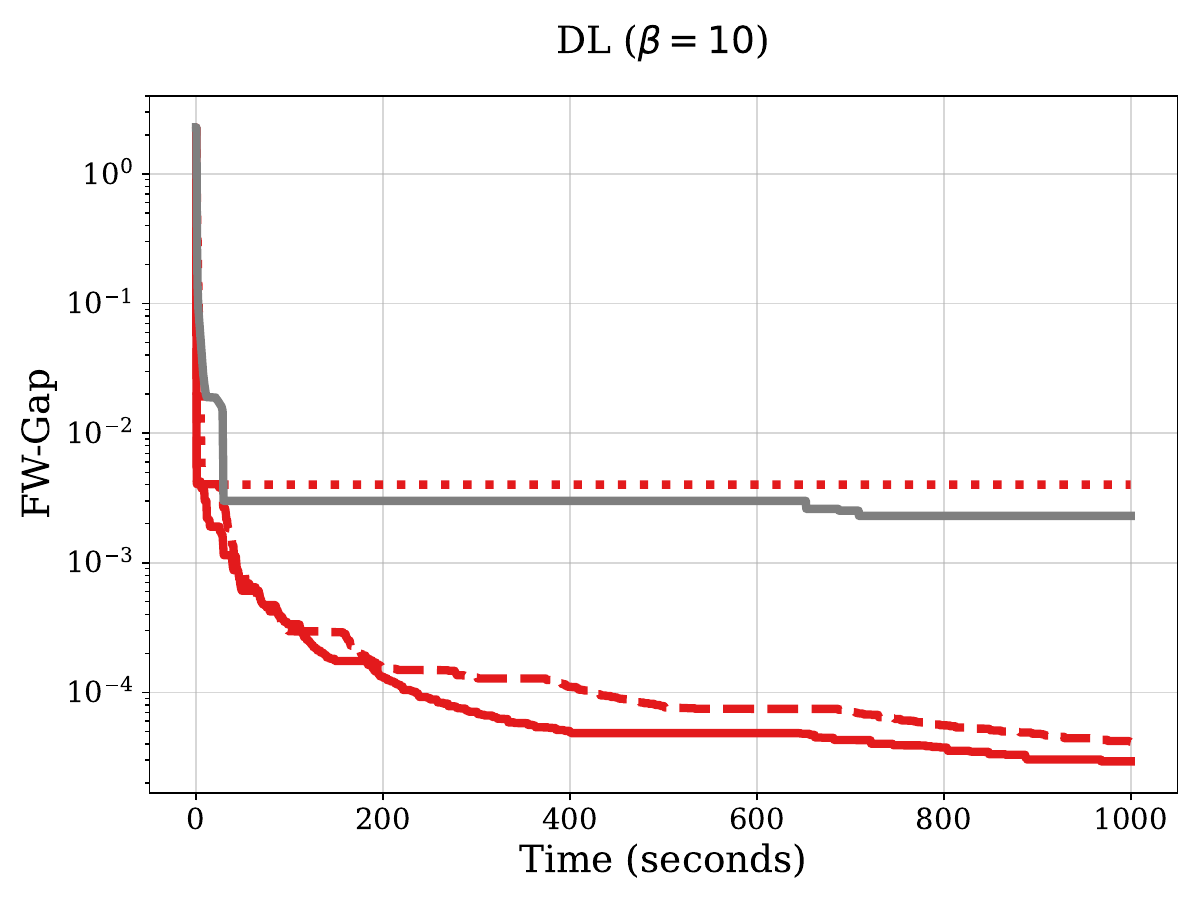}
        \label{subfig:dl10}
    \end{subfigure}
    ~
    \begin{subfigure}[b]{0.28\textwidth}
        \centering
        \includegraphics[width=\linewidth]{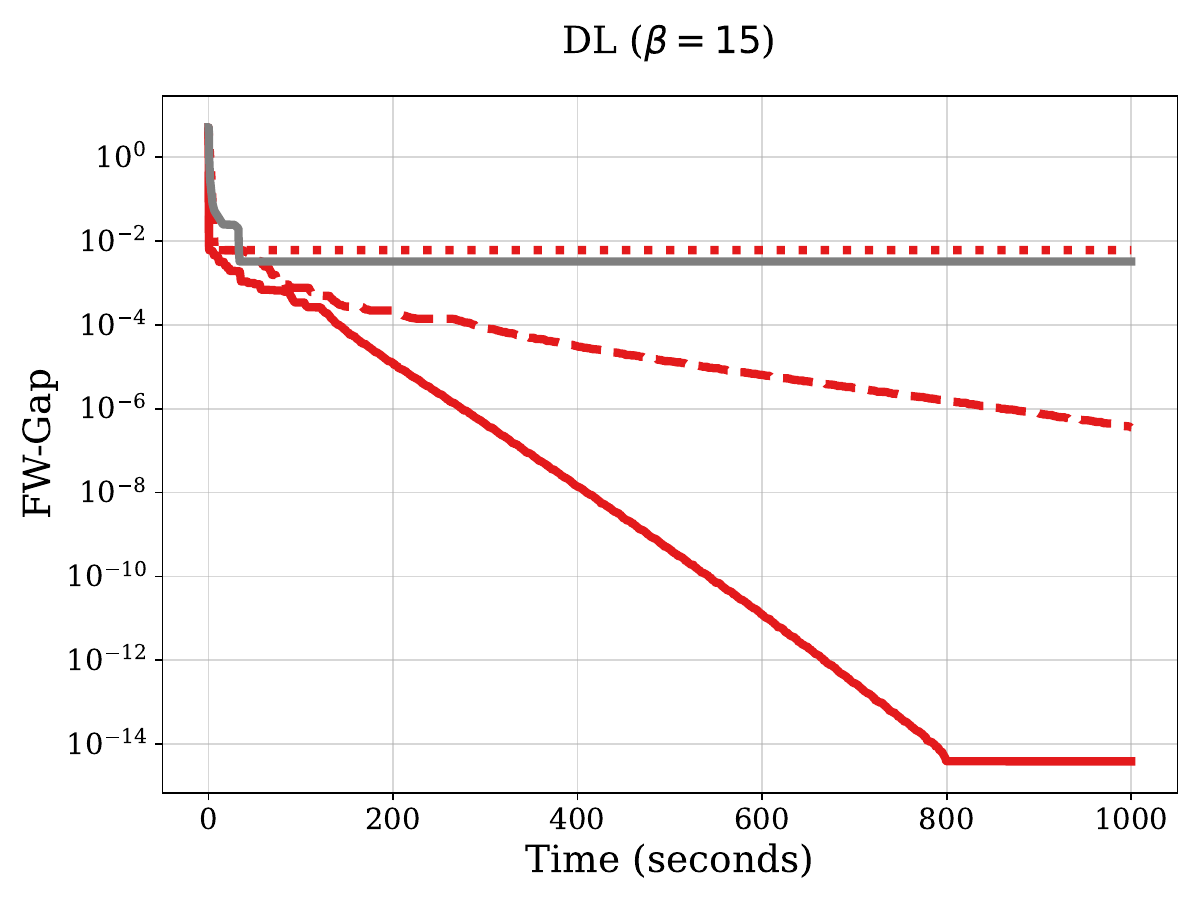}
        \label{subfig:dl15}
    \end{subfigure}

    \vspace{-1em}
    \begin{subfigure}[b]{0.28\textwidth}
        \centering
        \includegraphics[width=\linewidth]{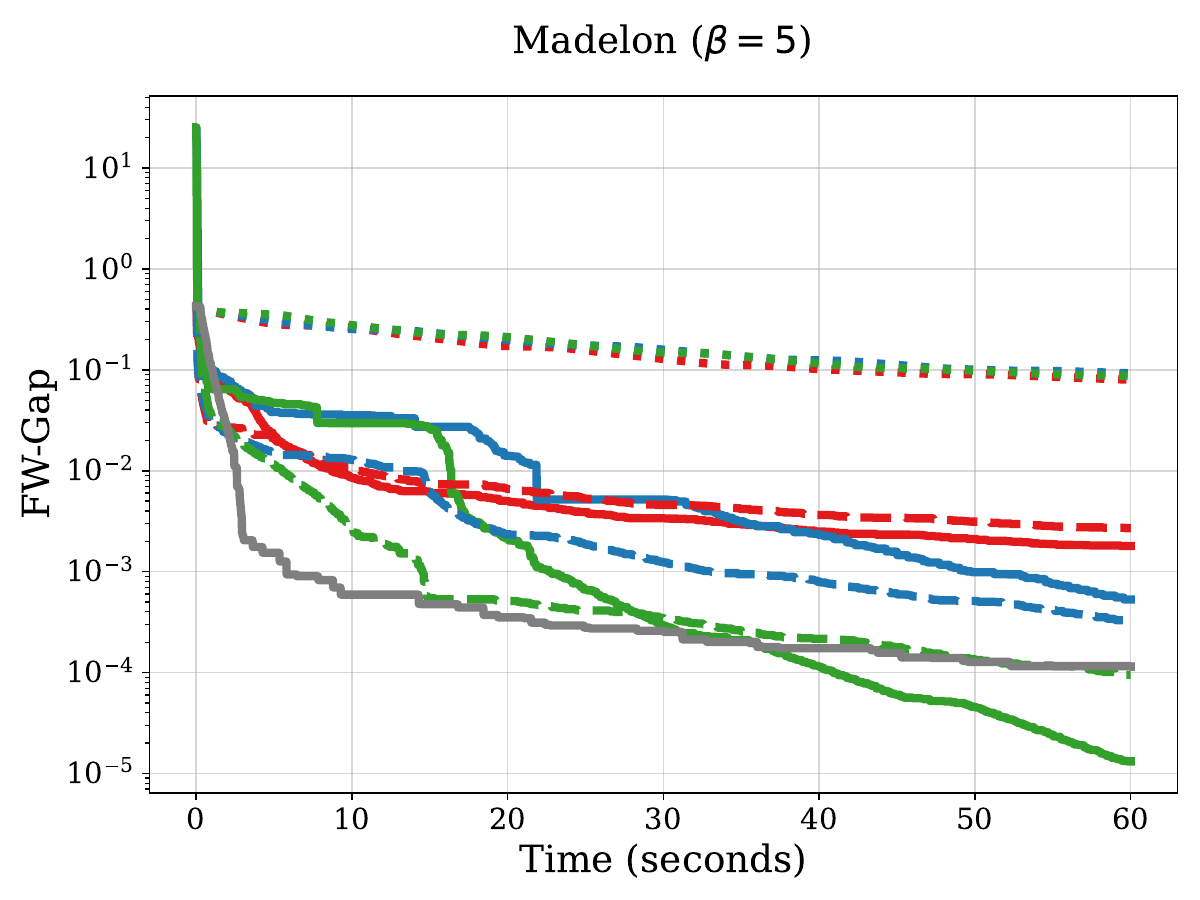}
        \label{subfig:madelon5}
    \end{subfigure}
    ~
    \begin{subfigure}[b]{0.28\textwidth}
        \centering
        \includegraphics[width=\linewidth]{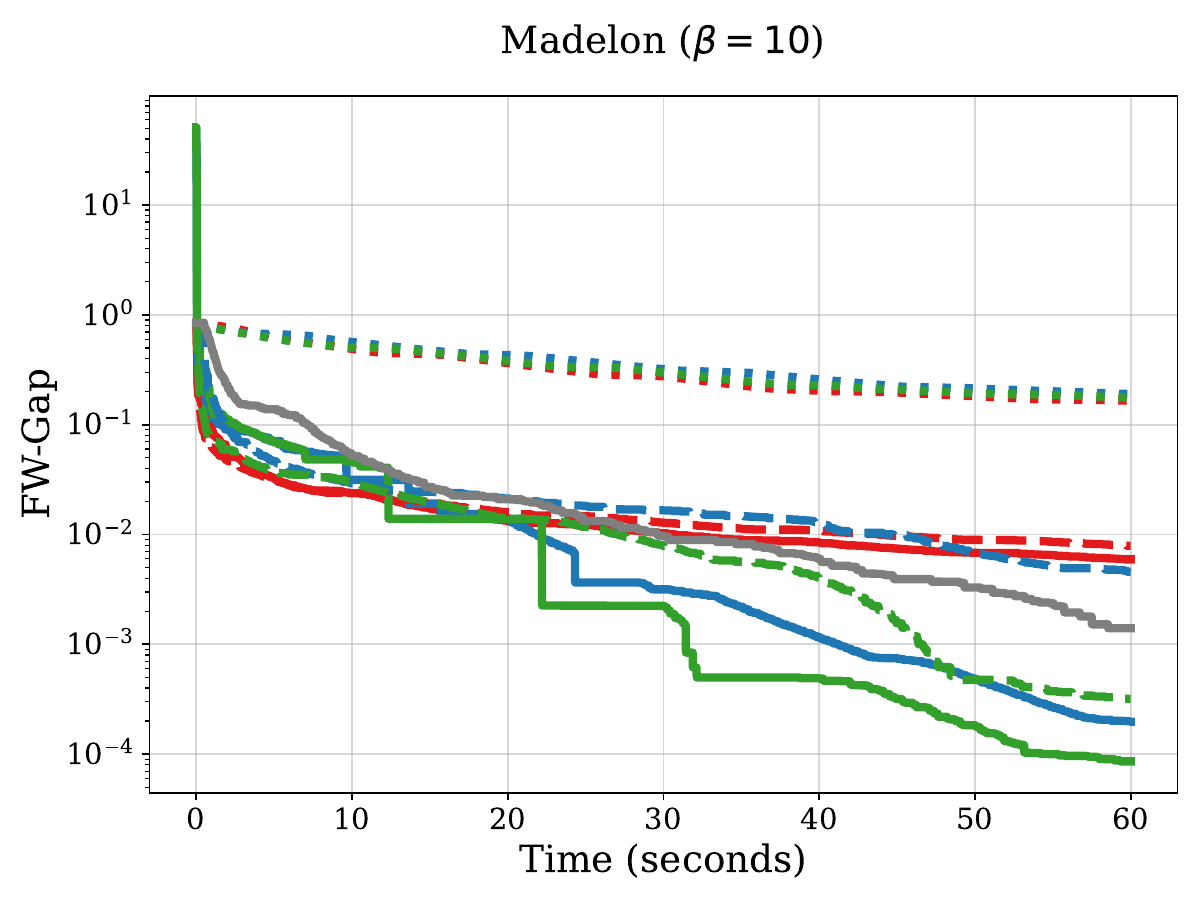}
        \label{subfig:madelon10}
    \end{subfigure}
    ~
    \begin{subfigure}[b]{0.28\textwidth}
        \centering
        \includegraphics[width=\linewidth]{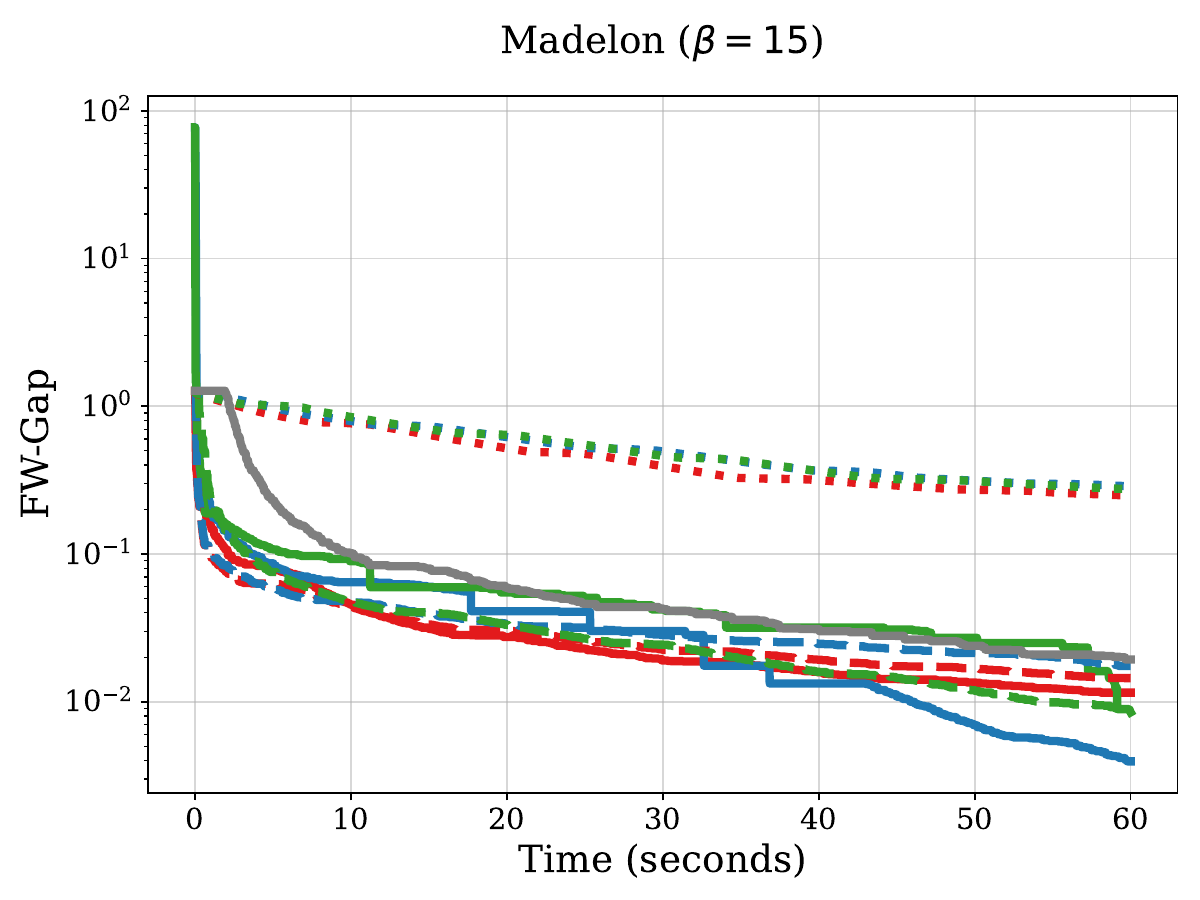}
        \label{subfig:madelon15}
    \end{subfigure}

    \vspace{-1em}
    \begin{subfigure}[b]{0.28\textwidth}
        \centering
        \includegraphics[width=\linewidth]{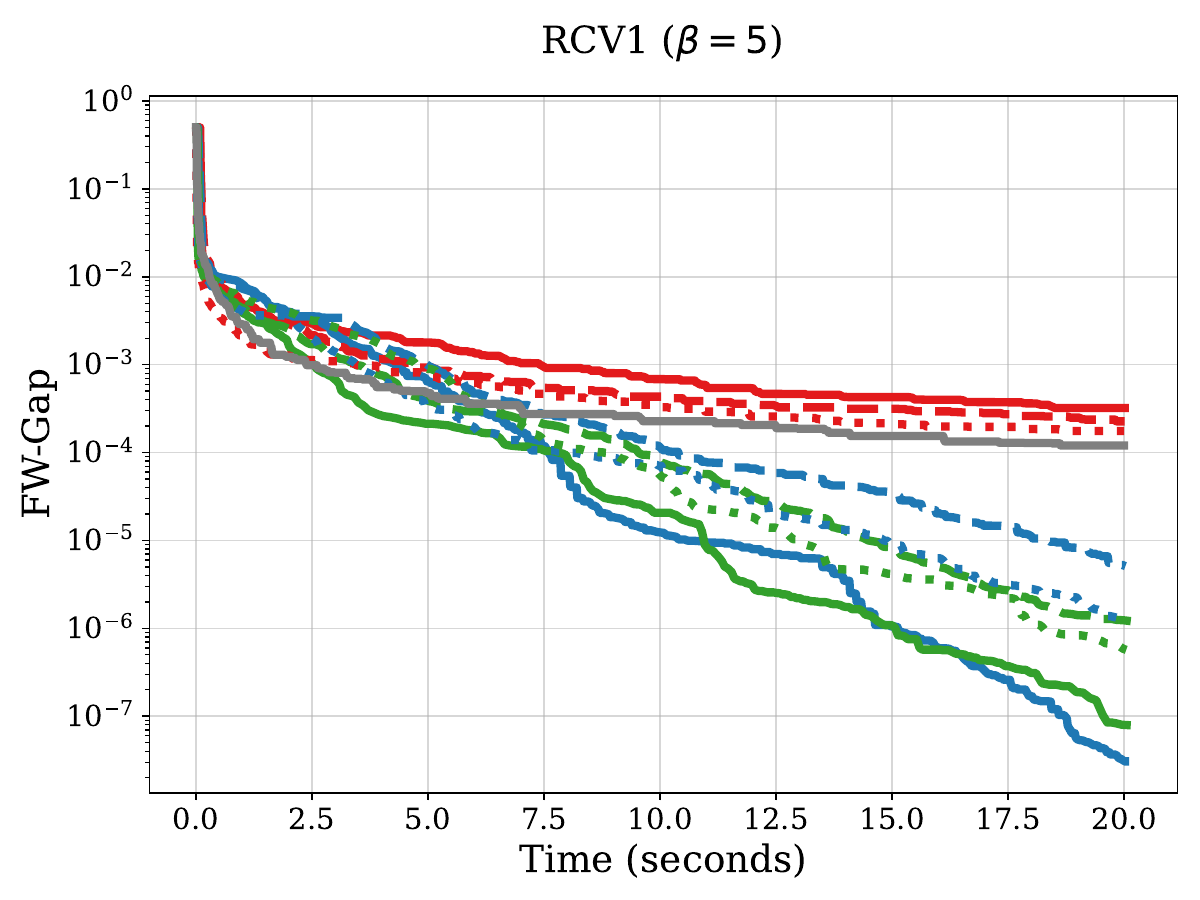}
        \label{subfig:rcv15}
    \end{subfigure}
    ~
    \begin{subfigure}[b]{0.28\textwidth}
        \centering
        \includegraphics[width=\linewidth]{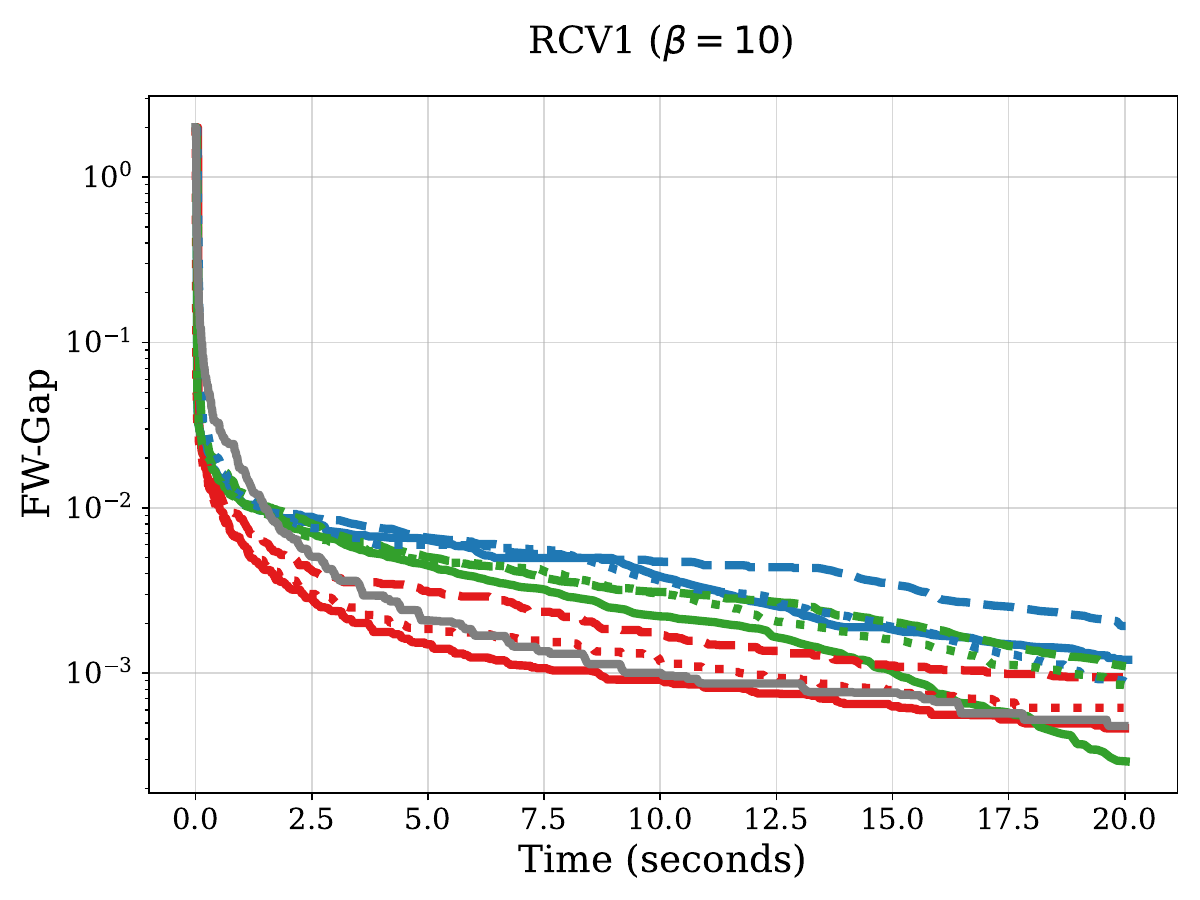}
        \label{subfig:rcv110}
    \end{subfigure}
    ~
    \begin{subfigure}[b]{0.28\textwidth}
        \centering
        \includegraphics[width=\linewidth]{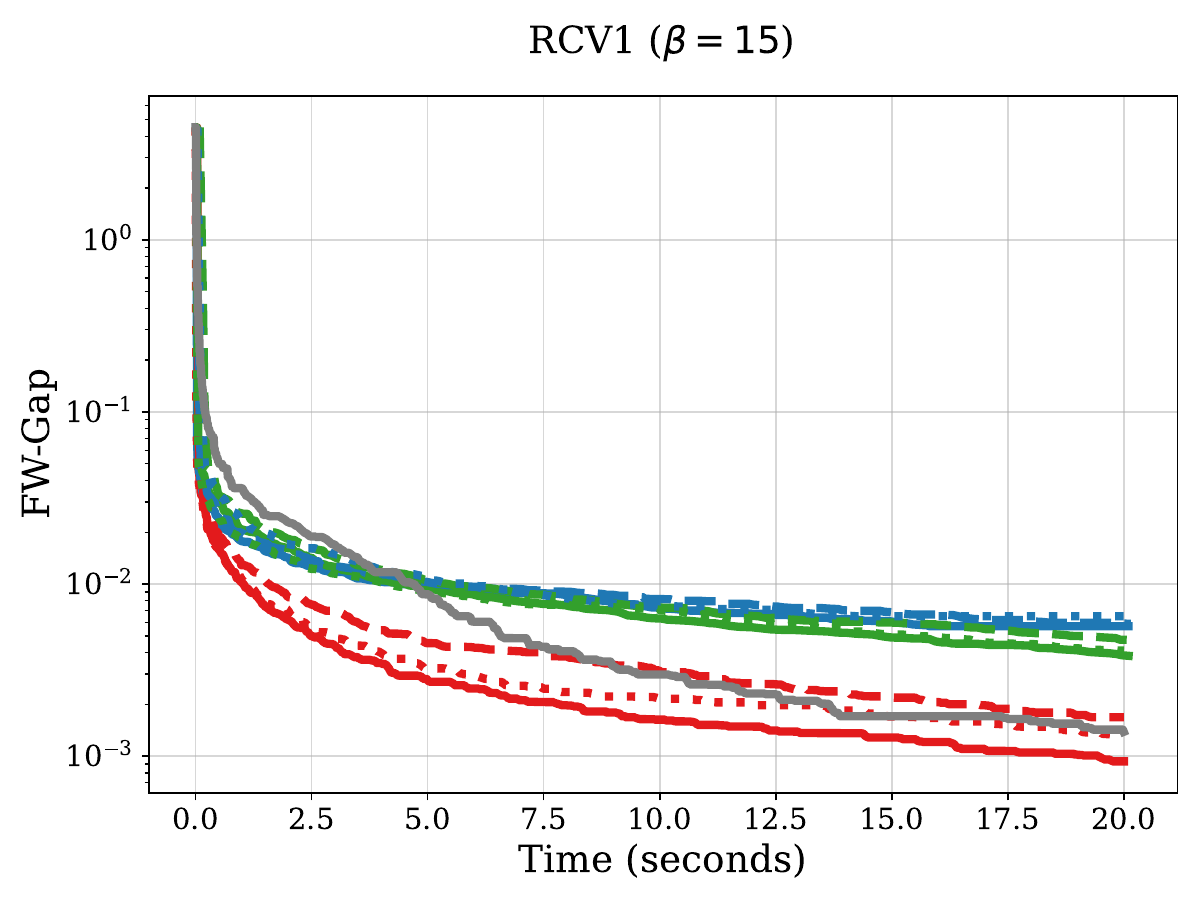}
        \label{subfig:rcv115}
    \end{subfigure}

    \vspace{-1em}
    \begin{subfigure}[b]{0.28\textwidth}
        \centering
        \includegraphics[width=\linewidth]{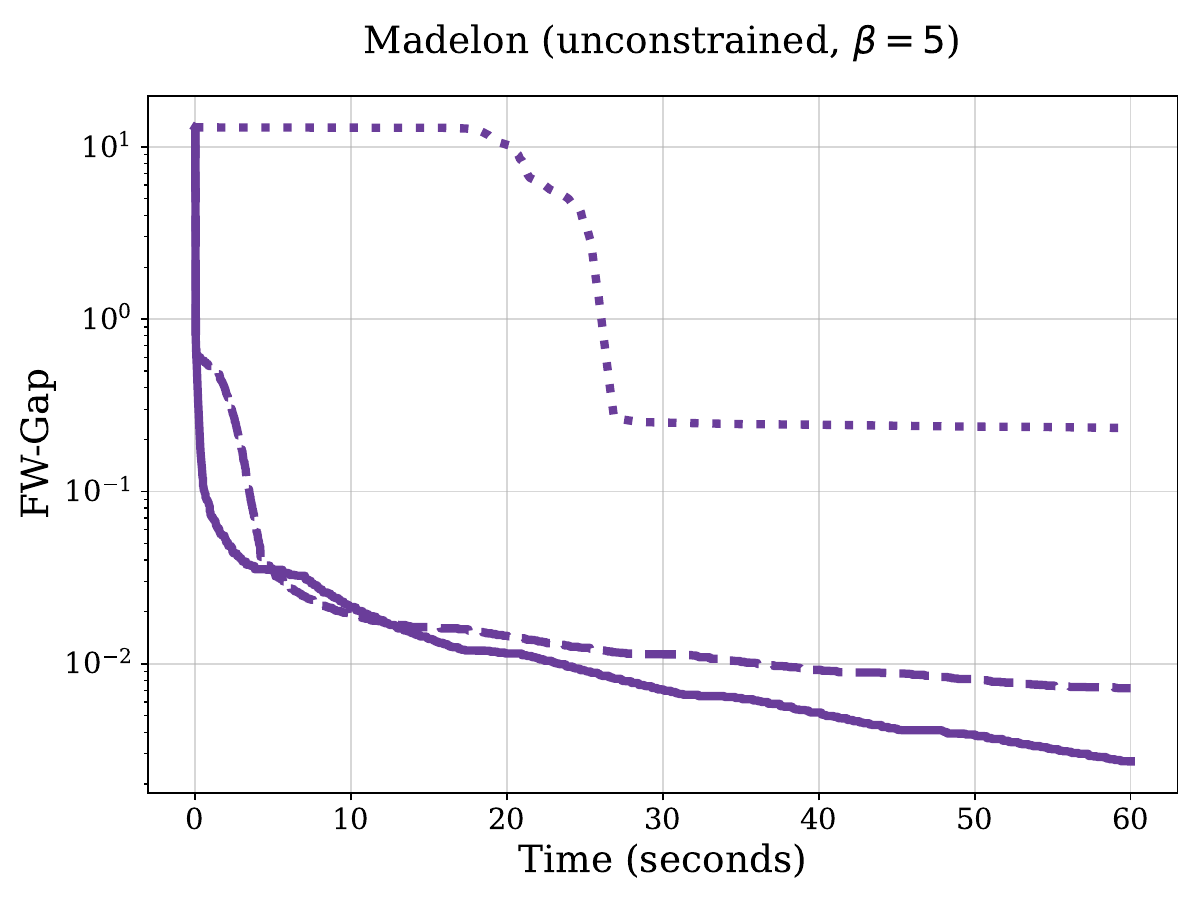}
        \label{subfig:un-madelon5}
    \end{subfigure}
    ~
    \begin{subfigure}[b]{0.28\textwidth}
        \centering
        \includegraphics[width=\linewidth]{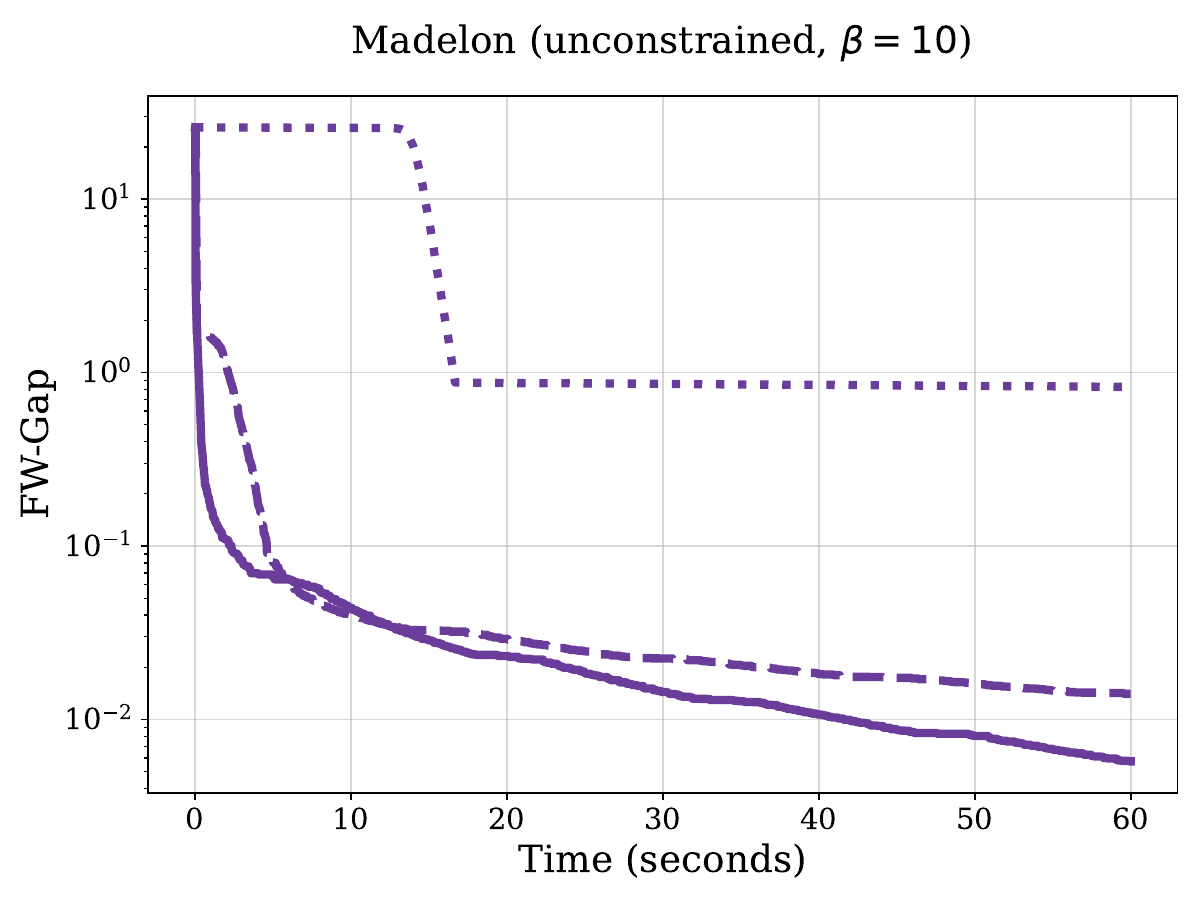}
        \label{subfig:un-madelon10}
    \end{subfigure}
    ~
    \begin{subfigure}[b]{0.28\textwidth}
        \centering
        \includegraphics[width=\linewidth]{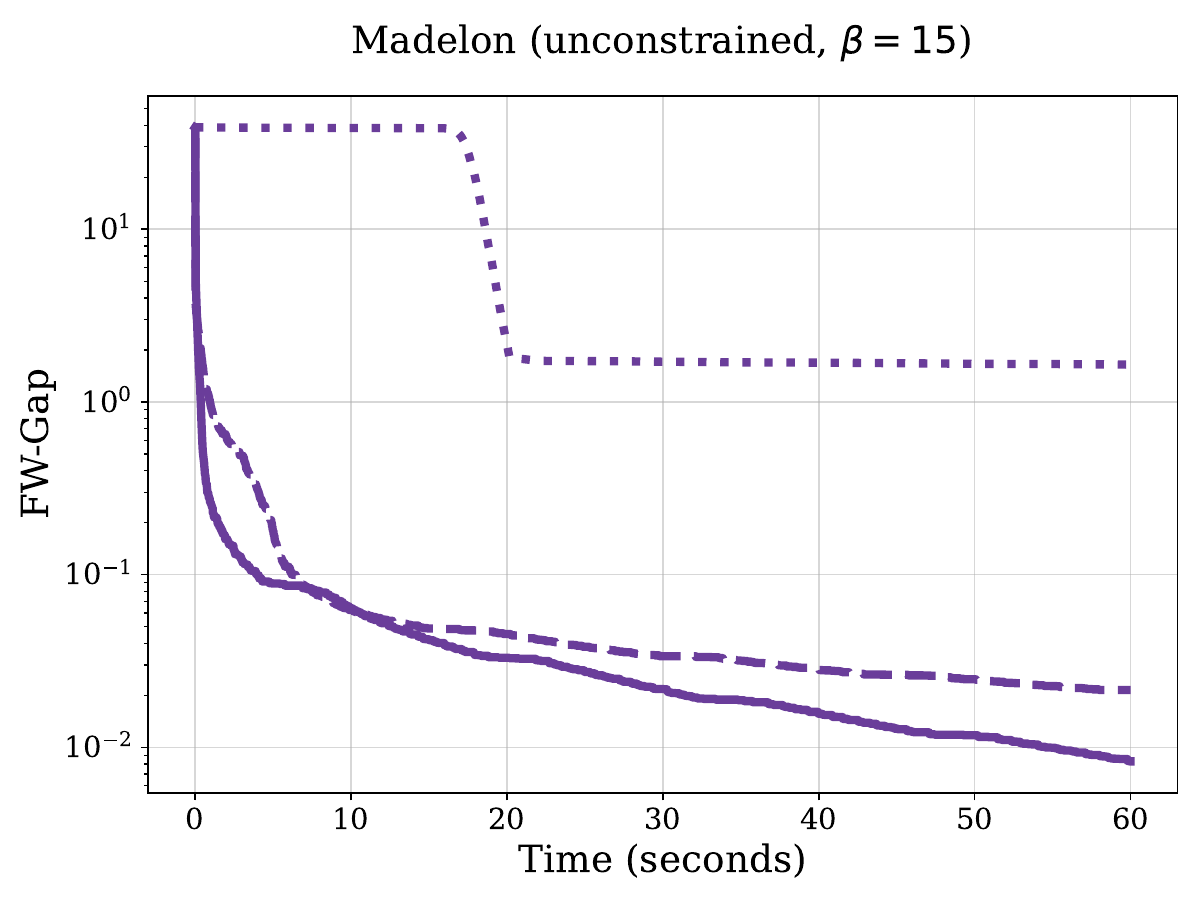}
        \label{subfig:un-madelon15}
    \end{subfigure}

    \vspace{-1em}
    \begin{subfigure}[b]{0.28\textwidth}
        \centering
        \includegraphics[width=\linewidth]{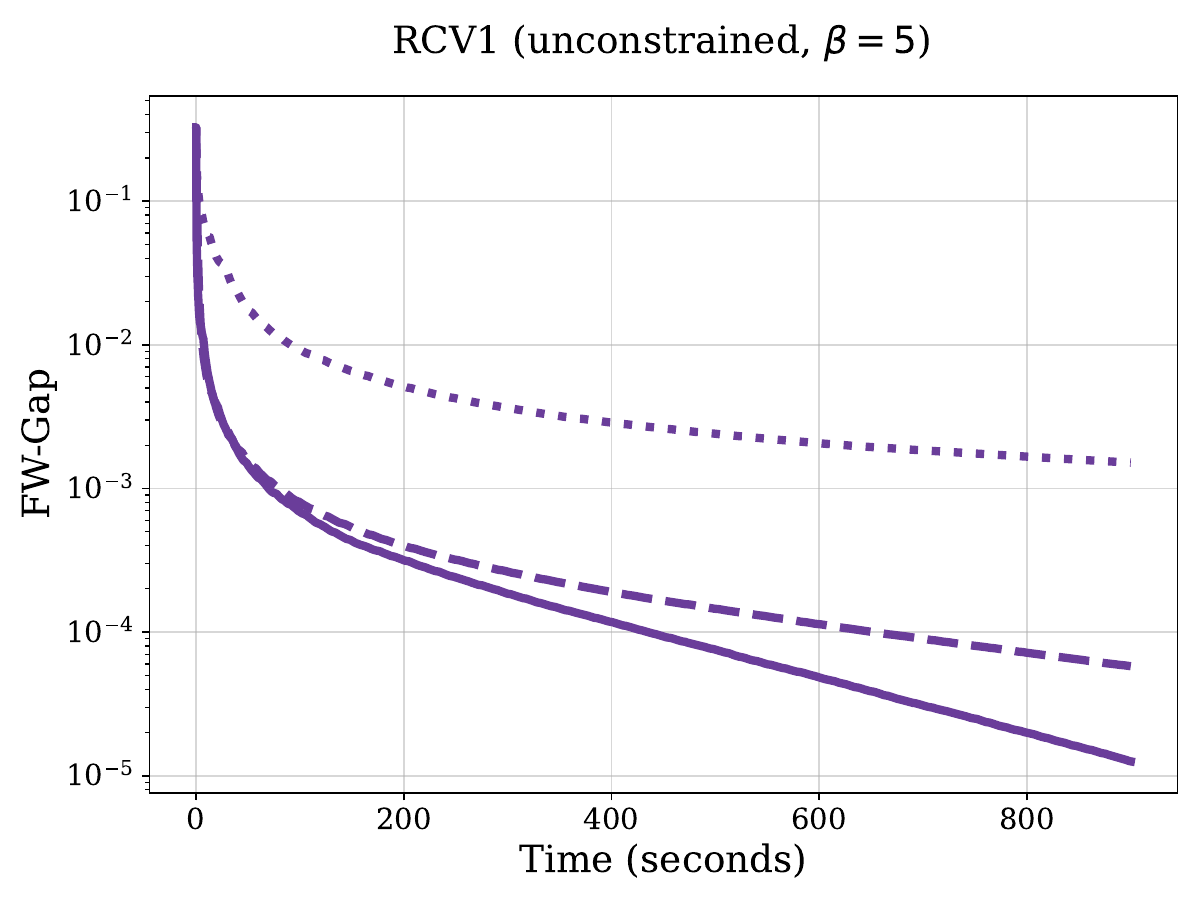}
        \label{subfig:un-rcv15}
    \end{subfigure}
    ~
    \begin{subfigure}[b]{0.28\textwidth}
        \centering
        \includegraphics[width=\linewidth]{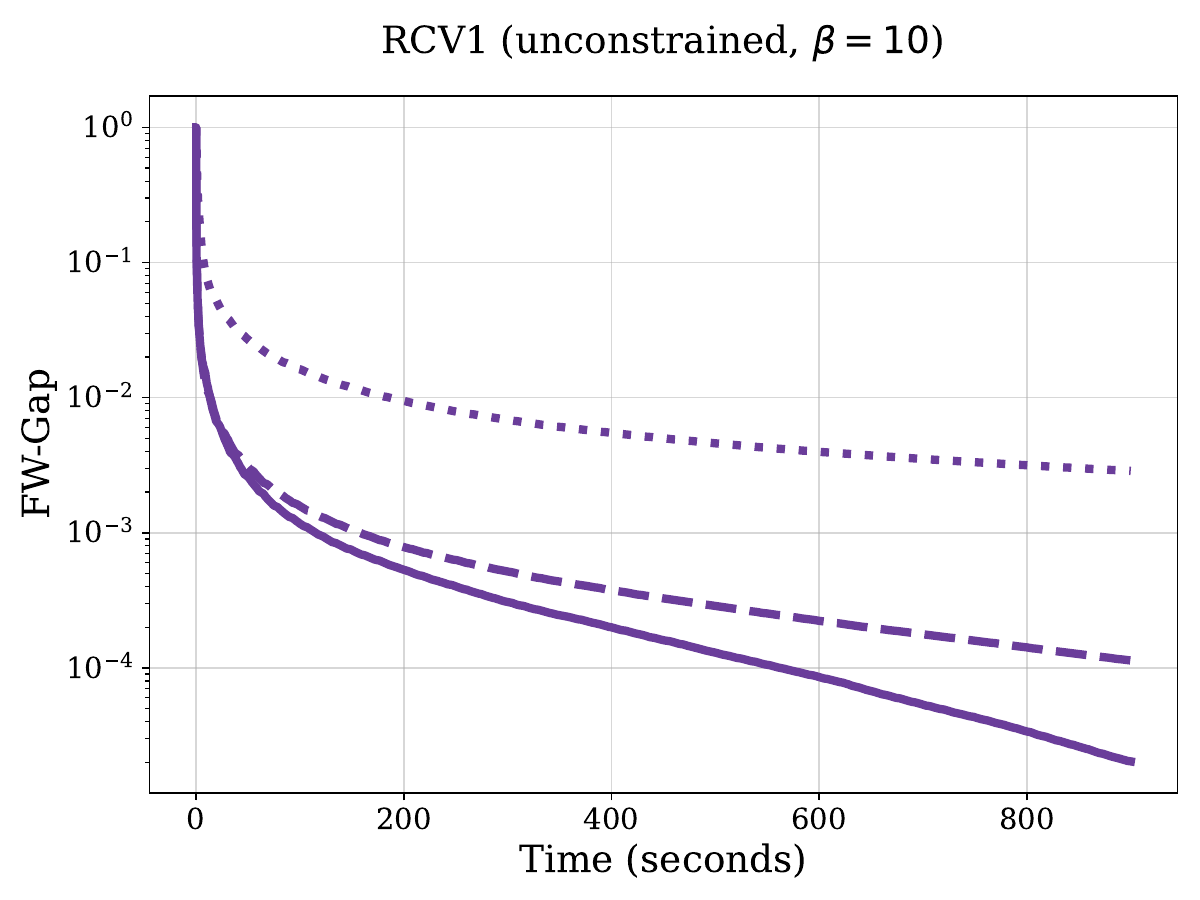}
        \label{subfig:un-rcv110}
    \end{subfigure}
    ~
    \begin{subfigure}[b]{0.28\textwidth}
        \centering
        \includegraphics[width=\linewidth]{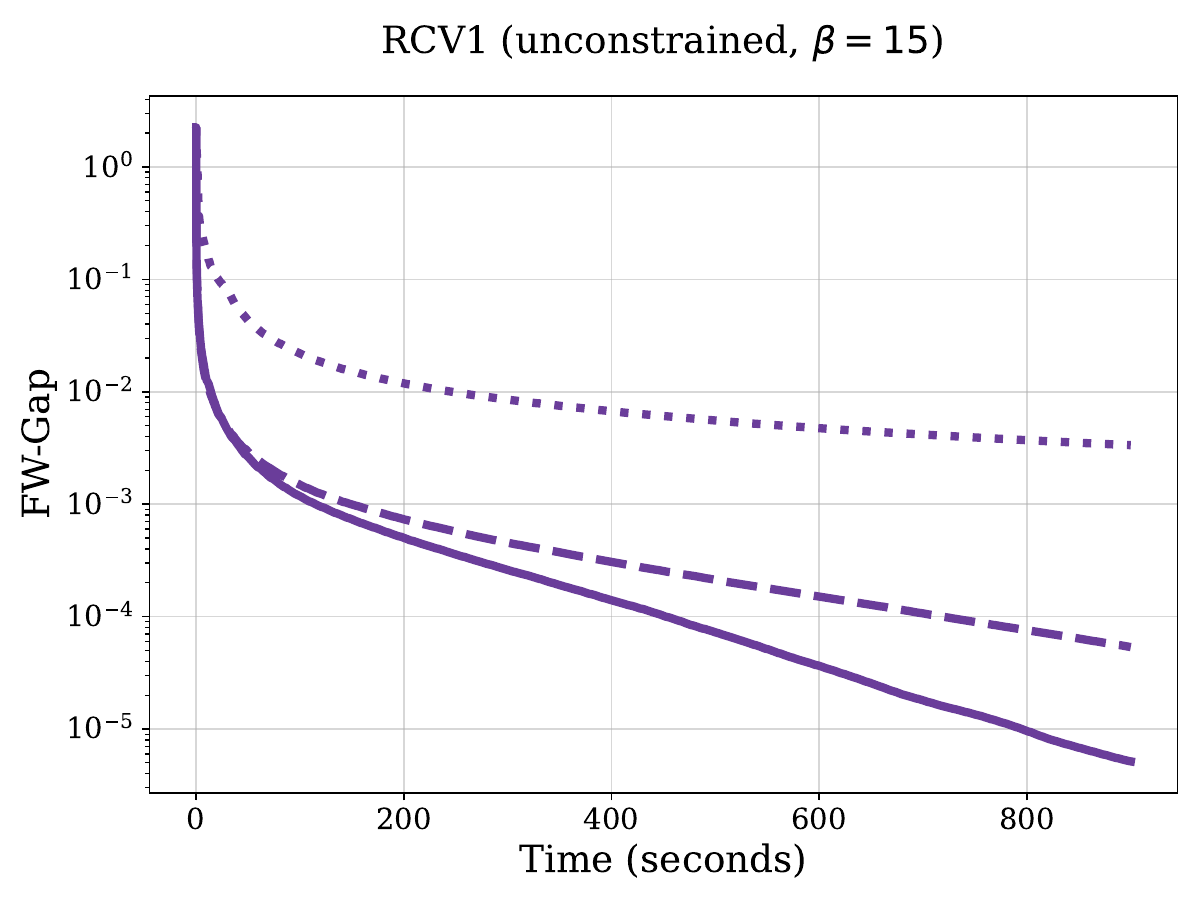}
        \label{subfig:un-rcv115}
    \end{subfigure}

    \vspace{-1em}
    \begin{subfigure}[b]{0.24\textwidth}
        \centering
        \includegraphics[width=\linewidth]{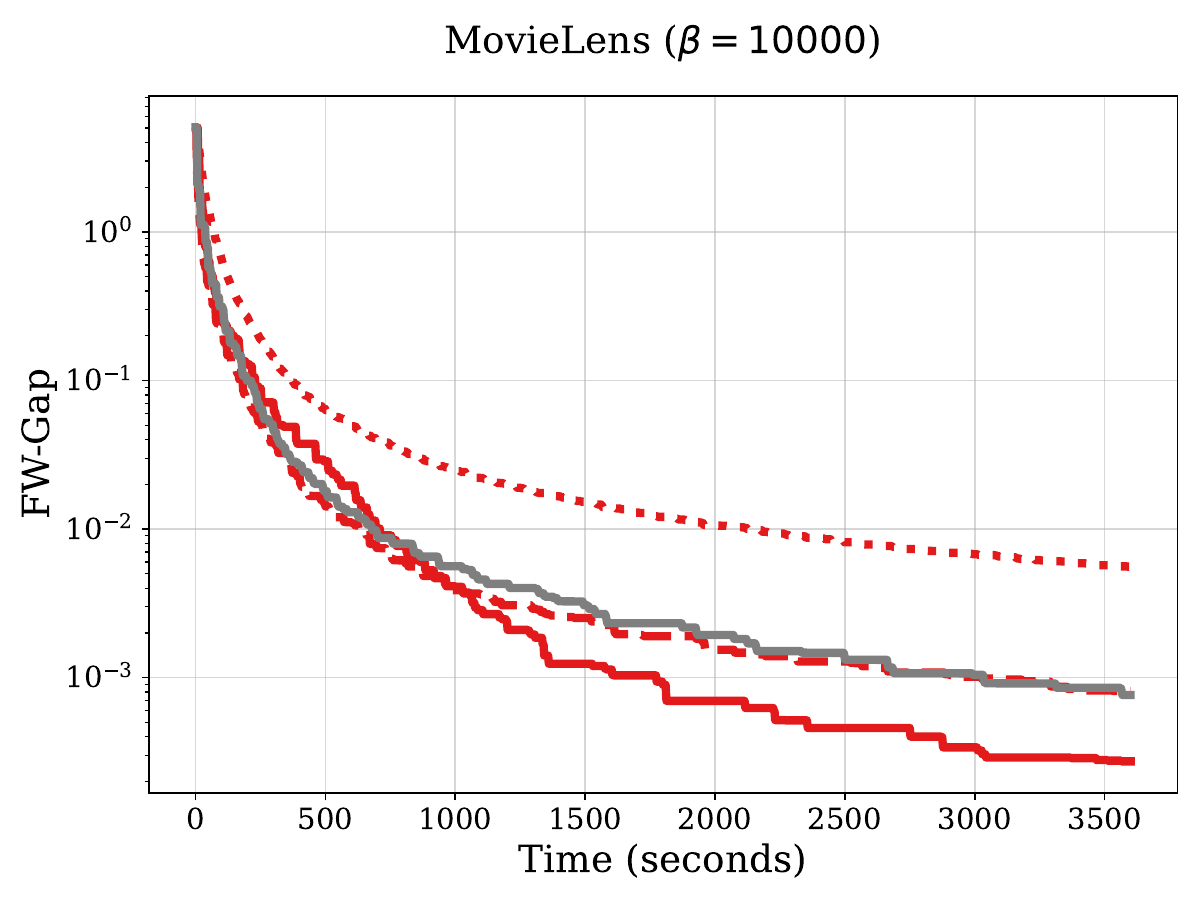}
        \label{subfig:movie10000}
    \end{subfigure}
    \hfill
    \begin{subfigure}[b]{0.24\textwidth}
        \centering
        \includegraphics[width=\linewidth]{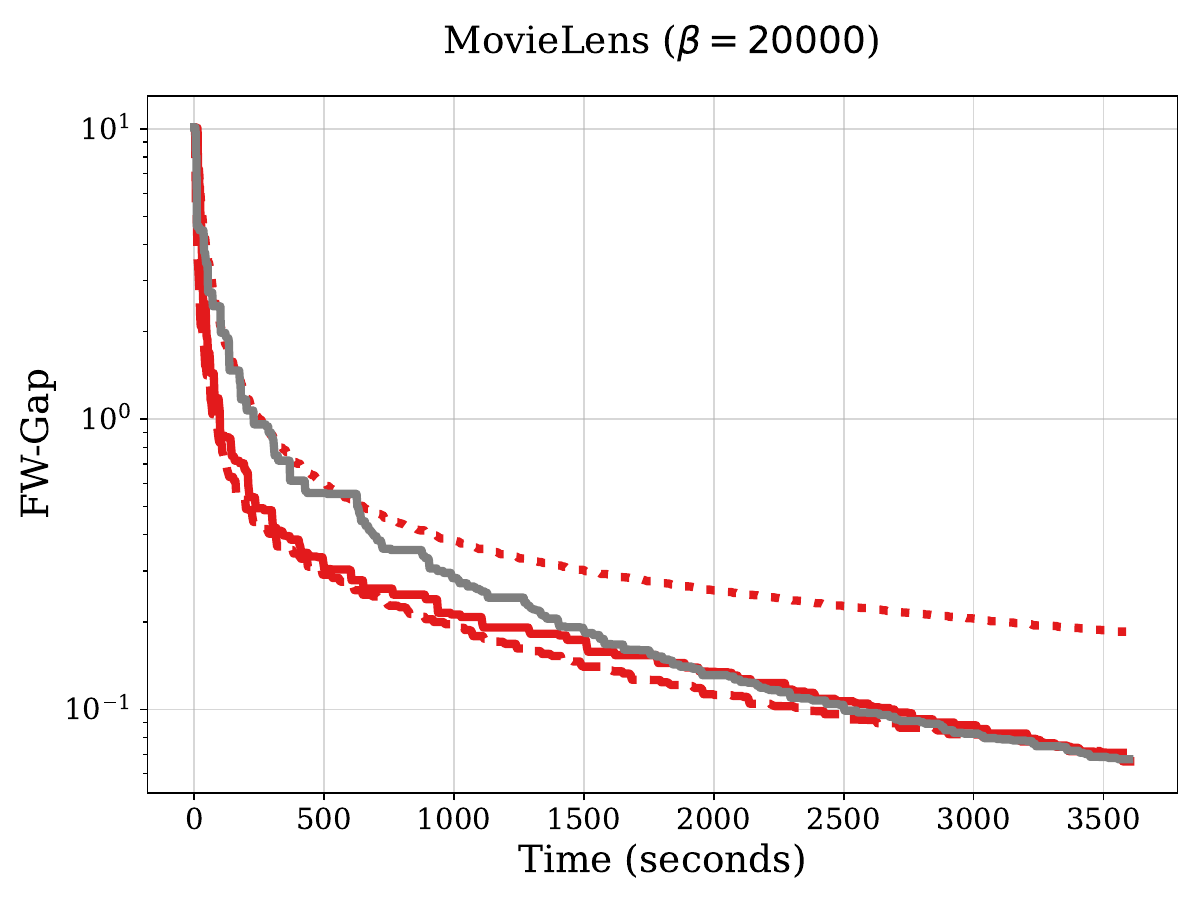}
        \label{subfig:movie20000}
    \end{subfigure}
    \hfill
    \begin{subfigure}[b]{0.24\textwidth}
        \centering
        \includegraphics[width=\linewidth]{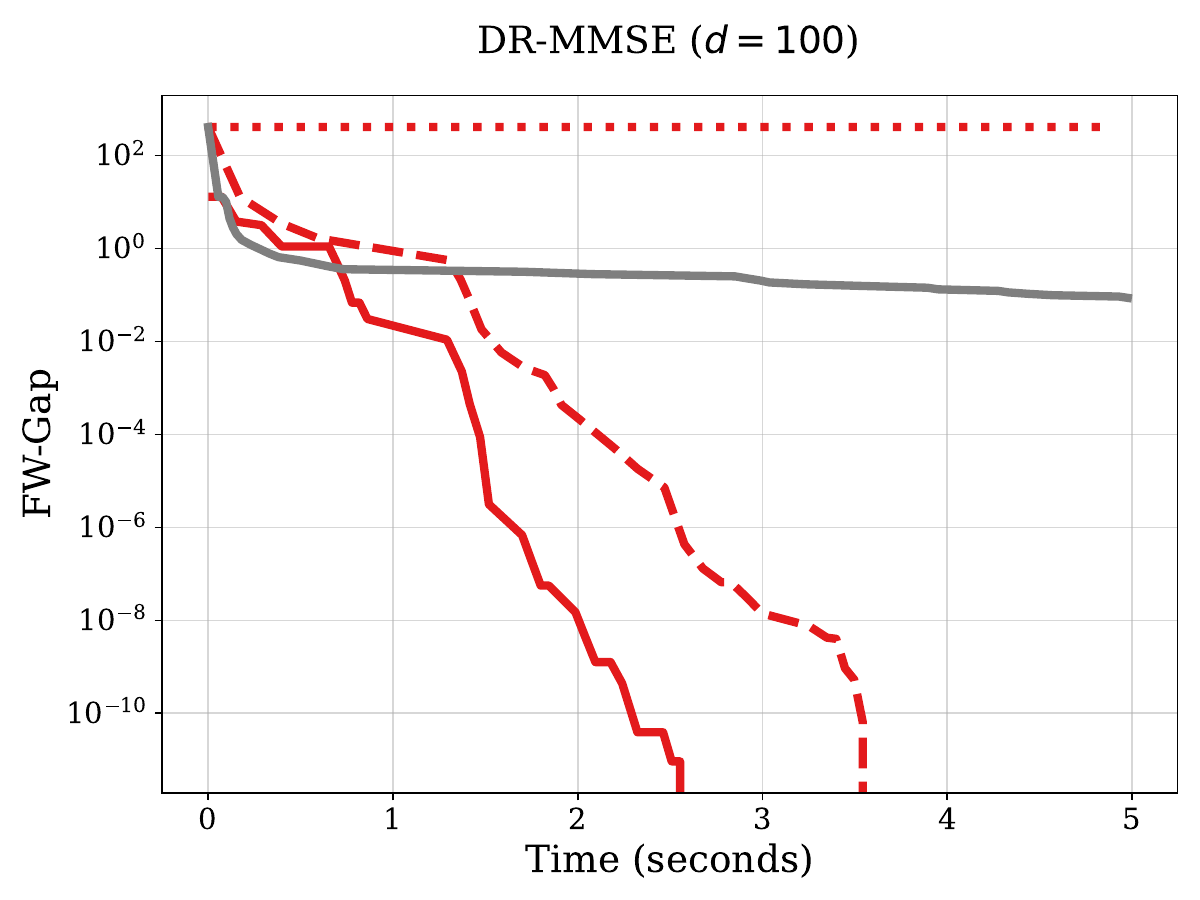}
        \label{subfig:DR100}
    \end{subfigure}
    \hfill
    \begin{subfigure}[b]{0.24\textwidth}
        \centering
        \includegraphics[width=\linewidth]{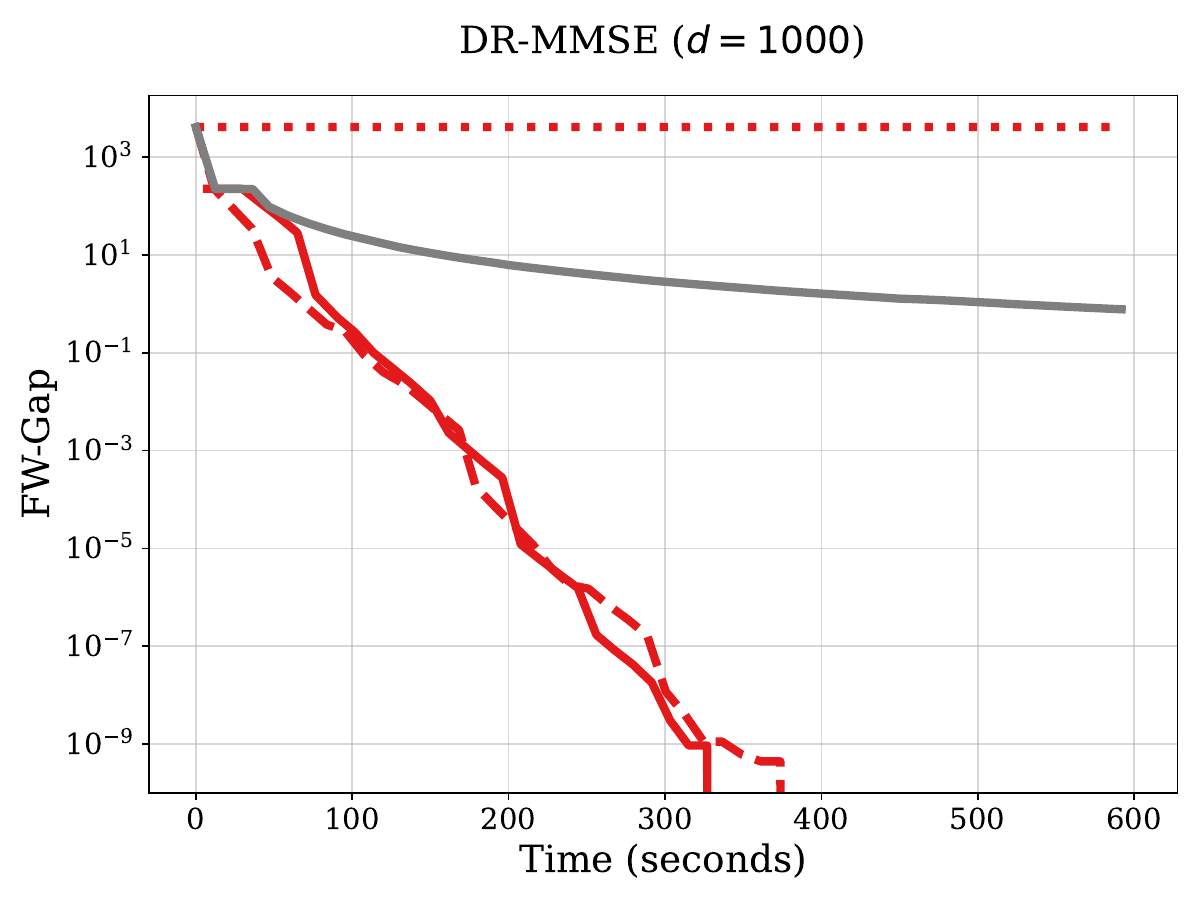}
        \label{subfig:DR1000}
    \end{subfigure}

    \vspace{-2em}
    \begin{subfigure}[b]{1\textwidth}
        \centering
        \includegraphics[width=\linewidth]{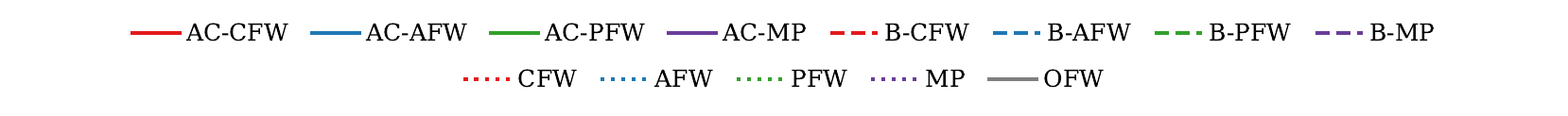}
    \end{subfigure}
    \vspace{-2.5em}
    \caption{Frank-Wolfe gap as a function of wall-clock time across all benchmark problems. Rows correspond, from top to bottom, to dictionary learning, constrained logistic regression on Madelon and RCV1, unconstrained logistic regression on Madelon and RCV1, and a final row comprising constrained Huber regression and distributionally robust minimum mean square error estimation. }
    \label{fig:Experiment1}
\end{figure*}

\section*{Acknowledgements}
This work was supported by the NSF CAREER ECCS-2541066.
\bibliographystyle{myabbrvnat}
\bibliography{bib} 

\appendix
\include{appendix}


\end{document}